\documentclass[12pt,a4paper,oneside,reqno]{amsart}
\usepackage{fancyhdr}
\usepackage{amssymb,mathrsfs,amsmath,amsthm}
\usepackage[numbers,sort&compress]{natbib}
\usepackage{authblk}
\usepackage{url}
\usepackage[numbers]{natbib}
\usepackage{setspace}
\usepackage{mathtools}
\usepackage{comment}
\usepackage[inline]{enumitem}
\usepackage[
  colorlinks=true,
  linkcolor=blue,
  citecolor=blue,
  urlcolor=blue,
  pagebackref,
  pdfauthor={Your Name},
  pdftitle={Your Document Title},
  pdfkeywords={keyword1, keyword2, keyword3}
]{hyperref}
\usepackage{comment}
\usepackage{hyperref}
\usepackage[margin=1in]{geometry} 

\newtheorem{thm}{Theorem}[section]
\newtheorem{cor}[thm]{Corollary}
\newtheorem{lem}[thm]{Lemma}
\newtheorem{prop}[thm]{Proposition}
\theoremstyle{definition}

\newtheoremstyle{boldremark}
  {10pt}  
  {10pt}   
  {}      
  {}       
  {\bfseries} 
  {.}      
  { }     
  {}  
\theoremstyle{boldremark}
\newtheorem{Remark}[thm]{\textbf{Remark}}
\numberwithin{equation}{section}
\newenvironment{mathclass}
  {Mathematics Subject Classification (2010):}
 
\newcommand{\R}{\mathbb{R}}
\newcommand{\C}{\mathbb{C}}
\newcommand{\N}{\mathbb{N}}

\renewcommand{\keywords}[1]{%
  \par\noindent
  Keywords: #1
  \par
}

\def\R {\mathbb{R}}

\newcommand{\be}{\begin{equation}}
\newcommand{\ee}{\end{equation}}
\newcommand{\bea}{\begin{eqnarray}}
\newcommand{\eea}{\end{eqnarray}}
\newcommand{\Bea}{\begin{eqnarray*}}
\newcommand{\Eea}{\end{eqnarray*}}
\newcommand{\bt}{\begin{Theorem}}
\newcommand{\et}{\end{Theorem}}

\newcommand{\bpr}{\begin{Proposition}}
\newcommand{\epr}{\end{Proposition}}

\newcommand{\bl}{\begin{Lemma}}
\newcommand{\el}{\end{Lemma}}
\newcommand{\bi}{\begin{itemize}}
\newcommand{\ei}{\end{itemize}}

\newtheorem{Definition}{Definition}[section]
\newtheorem{Theorem}[Definition]{Theorem}
\newtheorem{Lemma}[Definition]{Lemma}
\newtheorem{Proposition}[Definition]{Proposition}

\title
{Fractional nonlinear Schr\"odinger and Hartree equations  in modulation spaces}
\author{Divyang G. Bhimani}
\address{Divyang G. Bhimani\\Department of Mathematics\\Indian Institute of Science Education and Research\\ Pune 411008\\India}
\email{divyang.bhimani@iiserpune.ac.in}

\author{Diksha Dhingra}
\address{Diksha Dhingra \\Department of Mathematics\\ Indian Institute of Technology\\ Indore, 452020\\ India}
\email{dikshadd1996@gmail.com}

\author{Vijay Kumar Sohani}
\address{Vijay Kumar Sohani\\ Department of Mathematics\\ Indian Institute of Technology\\ Indore 452020\\ India}
\email{vsohani@iiti.ac.in}
\fancypagestyle{plain}{\fancyhf{} }
\begin{document}
\date{}
\maketitle{}
\begin{center}
DIVYANG G. BHIMANI, DIKSHA DHINGRA, VIJAY KUMAR SOHANI
\end{center}
\begin{abstract}
We establish global well-posedness
for the mass subcritical  nonlinear fractional Schr\"odinger equation 
    $$iu_t - (-\Delta)^\frac{\beta}{2} u+F(u)=0$$
having radial initial data in modulation spaces  $M^{p,\frac{p}{p-1}}(\mathbb R^n)$ for $n \geq 2, p>2$ and $p$ sufficiently close to $2.$  The nonlinearity $F(u)$ is  either of power-type  $F(u)=\pm (|u|^{\alpha}u)\; (0<\alpha<2\beta / n)$ or Hartree-type $(|x|^{-\nu} \ast |u|^{2})u \; (0<\nu<\min\{\beta,n\}).$ Our order of dispersion $\beta$ lies in $(2n/ (2n-1), 2).$
\end{abstract}
\footnotetext{\begin{mathclass} Primary 35Q55, 35Q60, 35R11, 42B37; Secondary 35A01.
\end{mathclass}
}
\section{Introduction}
We study Cauchy problem for the fractional nonlinear Schr\"odinger equation (FNLS), namely,
\begin{equation}\label{FNLS}
\begin{cases}
    iu_t(x,t) - (-\Delta)^\frac{\beta}{2} u(x,t)+F(u)=0\\ u(x,0)=u_0(x)
\end{cases}(x , t ) \in \mathbb R^n \times \mathbb R 
\end{equation}where $u(x,t) \in \mathbb C,$ and $(-\Delta)^{\beta/2}$ denotes  the fractional Laplacian with $ \beta>0$. Here,  the nonlinearity is either of the power type
\begin{equation*}\label{Powertype}
    F(u)=\pm (|u|^{\alpha}u) \quad   (0<\alpha<\frac{2\beta}{n})  
\end{equation*}
or Hartree type
\begin{equation*}\label{Hartree}
    F(u)=(|\cdot|^{-\nu} \ast |u|^{2})u \quad (0<\nu< \min \{\beta,n\}),
\end{equation*}
where $\ast$ denotes the convolution in $\R^{n}$.
The model \eqref{FNLS} was introduced in the
theory of the fractional quantum mechanics by Laskin through an expansion of the Feynman path integral, transitioning from Brownian-like to L\'evy-like quantum mechanical paths \cite{Laskin1,Laskin2}. Equation \eqref{FNLS} with Hartree type nonlinearity, also known as Boson star equation, arises in the mean-field limit of large system of bosons, see \cite{Enno}. This is also referred as  Schr\"odindger-Newton equation or Choquard equation in another physical context \cite{Lieb}.

The goal of this article is to establish local well-posedness (LWP) and global well-posedness (GWP) for \eqref{FNLS} having initial  data in modulation spaces.
In the early 1980s  Feichtinger  \cite{Feih83} introduced a  class of Banach spaces,  the so called modulation spaces. We briefly  recall them here, see \cite{KassoBook, wang2011harmonic} for a thorough introduction. 
Let $\rho: \mathbb R^n \to [0,1]$  be  a smooth function satisfying   $\rho(\xi)= 1 \  \text{if} \ \ |\xi|_{\infty}\footnote{Define $|\xi|_{\infty}=\max\{ | \xi_i | : \xi= (\xi_1,..., \xi_n)\}.$}\leq \frac{1}{2} $ and $\rho(\xi)=
0 \  \text{if} \ \ |\xi|_{\infty}\geq  1$. Let  $\rho_k$ be a translation of $\rho,$ that is,
$ \rho_k(\xi)= \rho(\xi -k) \ (k \in \mathbb Z^n).$
Denote 
$$\sigma_{k}(\xi)= \frac{\rho_{k}(\xi)}{\sum_{l\in\mathbb Z^{n}}\rho_{l}(\xi)}\quad (k \in \mathbb Z^n).$$
The frequency-uniform decomposition operators can be  defined by 
$$\square_k = \mathcal{F}^{-1} \sigma_k \mathcal{F} \quad (k \in \mathbb Z^n)$$ 
where $\mathcal{F}$ and $\mathcal{F}^{-1}$ denote the Fourier and inverse Fourier transform respectively. The weighted modulation spaces  $M^{p,q}_s \ (1 \leq p,q \leq \infty, s \in \R)$ is defined as follows:
\begin{equation*}
    M^{p,q}_s= M^{p,q}_s(\R^n)= \left\{ f \in \mathcal{S}'(\R^n): \left|\left|f\right|\right|_{M^{p,q}_s}=  \left\| \|\square_kf\|_{L^p_x} (1+|k|)^{s} \right\|_{\ell^q_k}< \infty  \right\} . 
\end{equation*}
For $s=0,$ we write $M^{p,q}_0= M^{p,q}.$ For $p=q=2,$ modulation spaces coincide with Sobolev spaces, i.e. $M^{2,2}_s= H^s \ (s \in \R). $  For $p\in [1, \infty]$,  we denote $p'$ the H\"older conjugate, i.e.  $\frac{1}{p}+\frac{1}{p'}=1.$ Denote $X_{rad}$ the set of all radial functions in $X$.  We are now ready to state our main results in the next two subsections.

\subsection{Fractional nonlinear Schr\"odinger equations}
We consider \eqref{FNLS} with power type nonlinarity:
\begin{equation}\label{FNLSP}
\begin{cases}
    iu_t(x,t) - (-\Delta)^\frac{\beta}{2} u(x,t) \pm (|u|^{\alpha}u)(x,t)=0\\ u(x,0)=u_0(x)
\end{cases}(x , t ) \in \mathbb R^n \times \mathbb R. 
\end{equation}
\begin{thm}[Local well-posedness]\label{lwp} Let  $0 < \alpha < \frac{2\beta}{n}$ and $ \frac{2n}{2n-1} < \beta < 2$  for $n \geq 2.$  
  Assume $u_{0}\in L^2_{rad}+M^{\alpha+2,(\alpha+2)'}_{rad}.$ Then there exists $T^*=T^*(\|u_{0}\|_{L^2+M^{\alpha+2,(\alpha+2)'}},n,\alpha,\beta)>0$ and a unique maximal  solution $u$ of  \eqref{FNLSP} such that 
    \begin{equation*}
    u\in \big(C([0,T^*),L^2_{rad})\cap L^{\frac{2\beta(\alpha+2)}{n\alpha}}([0,T^*),L^{\alpha+2}) \big)+C([0,T^*),M^{\alpha+2,(\alpha+2)'}_{rad}).
     \end{equation*}
      \end{thm}
       
The local solution established in Theorem \ref{lwp} can be extended to a global one under certain restriction on exponent $p$ \footnote{Refer to  Remark \ref{Whypso} for a comment on the restriction of $p$.}.  
To this end, we denote
\begin{equation}\label{pmax}p_{\max} := 
\begin{cases}
 2+\frac{2}{\alpha+1}-\frac{n\alpha}{\beta(\alpha+1)}\quad &\text{if} \  \alpha-\frac{n\alpha^2}{2\beta(\alpha+2)}-1+\frac{n\alpha}{2\beta}>0 \\ \alpha+2 \quad  & \text{otherwise}.  
\end{cases}
\end{equation}
\begin{thm}[global well-posedness]\label{gwp}  Let $0 < \alpha < \frac{2\beta}{n}$ and $ \frac{2n}{2n-1} < \beta < 2$  for $n \geq 2.$ Assume that  $u_{0}\in M^{p,p'}_{rad}$ for $p\in (2,p_{max}).$ Then the local solution constructed in Theorem \ref{lwp} of \eqref{FNLSP} extends globally and lies in
\begin{equation*}
    \big(C(\R,L^2_{rad}) \cap L^{\frac{2\beta(\alpha+2)}{n\alpha}}_{loc}(\R,L^{\alpha+2})\big)+ C(\R,M^{\alpha+2,(\alpha+2)'}_{rad}).
\end{equation*}
\end{thm}
The global well-posedness  for  \eqref{FNLS}  has been studied by many authors in Sobolev spaces, see Subsection \ref{pw} below. Modulation spaces have played a crucial role in the long-running  investigation of NLS \eqref{FNLS} ($\beta=2$) near Sobolev scaling criticality for the past two decades. We refer to, among others,   survey article by Ruzansky-Sugimoto-Wang
\cite{RuzSurvey} and  \cite{wang2011harmonic, Ruz4NLS, oh2018global, Kasso2009, KassoBook, Wang2006,Bhimani2016, BhimaniJFA2, BhimaniNorm, BhimaniHartree-Fock, Wang25,bhimani2023mixed}. Theorem \ref{gwp} is the first global well-posedness result for \eqref{FNLS} $(\beta \neq 2)$ for the large data in modulation spaces. It is  known that    $M^{p,p'} \neq H^{s}=M^{2,2}_s$ for $p \neq 2, s\in \mathbb R.$  It is interesting to note that Theorem \ref{gwp} is  applicable to  initial data of infinite $L^2-$norm, and also to the  data from $L^p_s-$Sobolev spaces  in view of the following sharp embedding when $p>2$ : 
\begin{equation*}\label{se}
  L^p_s\hookrightarrow  M^{p,p'} \hookrightarrow  L^p  \quad \text{for}\quad s > n\left(1-\frac{2}{p}\right) 
\end{equation*}
and 
\begin{equation}\label{HsMpq}
  H^s \hookrightarrow M^{p, p'} \quad \text{for}\quad   s> n\left(\frac{1}{2}-\frac{1}{p}\right).
 \end{equation}
\subsection{Fractional Hartree equations}
We now consider \eqref{FNLS} with  Hartree-type nonlinearity:
\begin{equation}\label{FNLSH}
\begin{cases}
    iu_t(x,t) - (-\Delta)^\frac{\beta}{2} u(x,t) \pm (|\cdot|^{-\nu} \ast |u|^{2})u=0\\ u(x,0)=u_0(x)
\end{cases}(x , t ) \in \mathbb R^n \times \mathbb R. 
\end{equation}

\begin{thm}[local well-posedness]\label{lwpHar}Let  $0<\nu<\min \{
\beta,n\}$ and $ \frac{2n}{2n-1} < \beta < 2$ for $n \geq 2.$  
  Assume $u_{0}\in L^2_{rad}+M^{\frac{4n}{2n-\nu},\frac{4n}{2n+\nu}}_{rad}.$ Then there exists $T^*=T^*(\|u_{0}\|_{L^2+M^{\frac{4n}{2n-\nu},\frac{4n}{2n+\nu}}},n,\nu,\beta)>0$ and a unique maximal  solution $u$ of  \eqref{FNLSH} such that 
    \begin{equation*}
    u\in \big(C([0,T^*),L^2_{rad})\cap L^{4\beta / {\nu}}([0,T^*),L^{4n/(2n-\nu)}) \big)+C([0,T^*),M^{\frac{4n}{2n-\nu},\frac{4n}{2n+\nu}}_{rad}).
     \end{equation*}
      \end{thm}
       The local solution established in Theorem \ref{lwpHar} can be extended to a global solution.  
\begin{thm}[global well-posedness]\label{gwpHar}  Let  $0<\nu<\min \{
\beta,n\}$ and $ \frac{2n}{2n-1} < \beta < 2$ for $n \geq 2.$ Assume that  $u_{0}\in M^{s,s'}_{rad}$ for $s\in (2,s_{max})$ where
\begin{equation}\label{qmax}s_{\max} := 
\frac{2n(4\beta-\nu)}{n(4\beta-\nu)-\nu(\beta-\nu)}. 
\end{equation}Then the local solution constructed in Theorem \ref{lwpHar} of  \eqref{FNLSH} extends globally and lies in
\begin{equation*}
    \big(C(\R,L^2_{rad}) \cap L^{4\beta / {\nu}}_{loc}(\R,L^{4n/(2n-\nu)})\big)+ C(\R,M^{\frac{4n}{2n-\nu},\frac{4n}{2n+\nu}}_{rad}).
\end{equation*}
\end{thm}

The study of Cauchy problem for  Hartree equations has long history  and has appeared in various physical settings e.g. white dwarfs and many particle physics etc.  \cite{Enno, Lieb}, \cite[Section 1.1]{BhimaniHartree-Fock}. In fact, it is   widely studied with Cauchy data in  $H^s$. 
Theorem \ref{gwpHar} complements this, see Subsection \ref{pw} below.

We now pursue further discussion on Theorem \ref{gwp} and Theorem \ref{gwpHar} in the following  subsection.
\subsection{Prior work for NLS and Hartree equations}\label{pw} FNLS \eqref{FNLS} exhibits a variety of rich structures, a few of which we highlight here. Formally, the solution of   \eqref{FNLSP} and \eqref{FNLSH} enjoy conservation of mass:
\begin{align}
M[u(t)] &= \int_{\mathbb{R}^n} |u(x, t)|^2 \, dx = M[u_0]. \label{mass}
\end{align}
Both the fractional NLS \eqref{FNLSP} and Hartree equation \eqref{FNLSH}  exhibit scaling symmetry:
\begin{itemize}
    \item If $u(x,t)$ is a solution of  \eqref{FNLSP} having initial data $u_{\lambda}(x,0)=\lambda^{-\frac{\beta}{\alpha}} u (  \lambda^{-1} x,0)$, then  $u_{\lambda}(x,t)= \lambda^{-\frac{\beta}{\alpha}} u (  \lambda^{-1} x,\lambda^{\beta} t)$ for $\lambda>0$ is
 also a solution of it.
 The $\dot{H}^{s}-$norm is invariant under this scaling  for $$s=s_{c}=\frac{n}{2}- \frac{\beta}{ \alpha},$$ so called critical Sobolev index for \eqref{FNLSP}.
 \item  If $u(x,t)$ solves \eqref{FNLSH}, then $u_{\lambda}$ defined as $u_{\lambda} (x, t) = \lambda^{\frac{n-\nu + \beta}{2}} u( \lambda x, \lambda^\beta t)$ also solves \eqref{FNLSH} with scaled initial data.  The $\dot{H}^{s}-$norm is invariant under this scaling  for 
 $$s=s_{c}=\frac{\nu-\beta}{2},$$ so called critical Sobolev index for \eqref{FNLSH}.
\end{itemize}
The scaling provides heuristics indicating that  \eqref{FNLSP} and \eqref{FNLSH} are well-posedness in $H^s$ for $s>s_c$ (sub-critical) and ill-posed in $H^s$ for $s<s_c$ (super-critical). 
Now one may ask whether there exists a class of data out of $H^{s_c}$ for which it is still well-posed. Modulation spaces turned out to be a very useful tool in order to answer this question. In fact, $M^{2,1}_{s_c}\subset  H^{s_c}$, $M^{2,1}_s\not\subset H^{s_c}$ for $s<s_c$; and  $H^{s_c}\not\subset M^{p,p'}$ for $p>2$, $s_c<0$.   See e.g. \cite[Remark 2.1]{BhimaniJDE2},  \cite{Ruz4NLS, Wang2006} and the references therein. We note that the norm of modulation space does not  preserve proper scaling symmetry \footnote{In \cite{BhimaniNorm, oh2018global}, authors considered almost scaling criticality.}. However, Bhimani et al. \cite{BhimaniNorm, BhimaniJFA2}  proved   \eqref{FNLSP} and  \eqref{FNLSH} are ill-posed in $M^{p,q}_s$ under certain negative regularity.

We briefly mention some  known results (among others)
having power-type and Hartree nonlinearity in the subcritical regime with initial data in either  Sobolev or modulation space:\\ 
$-$\textbf{Classical case $\beta=2$:}
\begin{enumerate}
    \item Tsutsumi  \cite{YTsutsumi} proved  \eqref{FNLSP} is  GWP in $L^{2}$ for $\alpha \in (0,\frac{4}{n})$. 
    Zhao, Wang and Guo  \cite{Wang2006} proved  \eqref{FNLSP} is LWP in $M^{2,1}$ for $\alpha \in 2 \mathbb N.$ Later, it is extended in $M_s^{p,1} \ (1\leq p \leq \infty, s\geq 0).$  
    See \cite{Kasso2009}. While Wang and Hudzik \cite{wango}   
 proved  classical NLS  \eqref{FNLSP}  is GWP for small data in $M^{2,1}$.  Oh and Wang \cite{oh2018global}  proved GWP for 1D cubic NLS  \eqref{FNLSP} in $M^{2,q}$ for $2\leq q< \infty.$  Later, it is extended  \cite{Klaus} in $X$ where $X=M^{p,q}$ for $1/q \geq 1/p, 1\leq p \leq 2$ or $M^{p,q}_s$ for $s>1/q', 1\leq p \leq 2$  or $M_1^{p,1}$ for $2\leq p< \infty.$ 
   Chaichenets et al. in \cite{LeonidIn,leonidthesis} established  GWP for $1$D cubic NLS \eqref{FNLSP} in  $M^{p, p'}$ for $p$ sufficiently close to 2. 
   \item In the early 1970s,  Chadam and  Glassey    proved global existence  for classical Hartree equation in the energy space $H^1$ in their seminal work \cite{ChadamGlassey}. Since then many authors  have studied Cauchy problem for the Hartree equation, see e.g. influential  papers \cite{MiaoJPDE, MiaoJFA} by  Miao, Xu, and  Zhao. 
   
   Carles \cite{Carles} proved \eqref{FNLSH}  is GWP in $L^2$ for $0<\nu< \min\{2, n\}.$ Moreover, it is GWP in $L^2\cap \mathcal{F}L^1$ for  $0<\nu< \min\{2, n/2\}$ where $\mathcal{F}L^1=\{f \in \mathcal{S'}: \hat{f}\in L^1\}$ (Wiener algebra). 
Later, Bhimani \cite[Theorem 1.1]{bhimani2016cauchy}  proved it is 
GWP  in $M^{1,1}$ and further extended by Manna \cite[Theorem 1.1]{manna2017modulation}  in $M^{p,p}$ for $1 \leq p < \frac{2n}{n+\nu}.$ Moreover, Bhimani  \cite{DGBJDE} extended  the range of $p, q$, specifically, it is GWP in $M^{p,q}$ with $1 \leq p \leq 2, 1 \leq q <
\frac{2n}{n+\nu}.$ In \cite{RameshManna}, Manna established  global existence for   $1$D  classical Hartree equation for  $0 <\nu <1$ in $M^{p,p'}(\mathbb R)$ for $2<p$ and sufficiently small. We also mention that Kim, Lee and Seo \cite{DiracHartree}  established LWP for Dirac equations with a Hartree type nonlinearity in modulation spaces. While Bhimani, Grillakis and Okoudjou \cite{BhimaniHartree-Fock} also treated  Hartree   equation \eqref{FNLSH} in the presence  of harmonic potential. 
\end{enumerate}
$-$\textbf{Smaller dispersion case  $0<\beta<2$:}
\begin{enumerate}
    \item  Guo-Huo \cite{Bguo} proved 1D  cubic  \eqref{FNLSP} is GWP in $L^2$ for  $ \beta \in (1,2).$  Hong and Sire  \cite{HongSire} established  LWP for  \eqref{FNLSP} with $\beta \in (0,2)  \setminus \{1\}$ in $X$, where $X= H^s(\mathbb R)$ when $\alpha \geq 4, s>s_c$ or $H^s(\mathbb R^n)$ when $n\geq 2, \alpha \geq 2, s>s_c$. They also established small data GWP and scattering in some critical Sobolev spaces, specifically for   $X= H^{s_c}(\mathbb R)$ when $\alpha > 4$ or $H^{s_c}(\mathbb R^2)$ when $\alpha>3.$
Cho et al. \cite{Chodcds} proved LWP for 1D cubic  \eqref{FNLSP} with $1<\beta<2$  in $H^s$ for $s \geq \frac{2-\beta}{4}$.  Also see \cite{dinh2016}.  B\'enyi and Okoudjou \cite[Remark 1]{Kasso2009} proved LWP  for  \eqref{FNLSP} with $\beta \in [0,2]$ in $M^{p,1}_s$  $(s\geq 0, 1\leq p \leq \infty)$ for $\alpha \in  2 \mathbb N.$
     \item  Cho,  Hajaiej, Hwang, and  Ozawa  proved  GWP for \eqref{FNLSH} in $H^s$ for $s \geq \nu /2$ in their seminal work \cite{Cho}. 
     Bhimani, Okoudjou and Grillakis \cite{BhimaniHartree-Fock}   established GWP  for  \eqref{FNLSH}  for radial initial data in  $M^{p,q}$  $(1 \leq p \leq 2, 1 \leq q \leq 
\frac{2n}{n+\nu})$  for  $0<\nu < \min \{\beta,\frac{n}{2}\}$ and  $ \frac{2n}{2n-1} < \beta < 2$ with $n \geq 2.$ 
\end{enumerate}      
$-$\textbf{Larger dispersion case $\beta> 2$:}
\begin{enumerate}
    \item Dinh \cite[Corollary 4]{dinh2016} proved GWP for  \eqref{FNLSP} in $L^2$ for $\alpha \in (0,\frac{2\beta}{n})$. 
 Seong \cite{SeongK} proved 1D cubic   \eqref{FNLSP} is GWP in $H^{s}(\mathbb R)$ for $s \geq -1/2$ for $\beta =4$. 
Ruzhansky, Wang and Zhang \cite{Ruz4NLS}  established small data GWP for fourth order $(\beta=4)$  \eqref{FNLSP} in some weighted modulation spaces.  Kato \cite[Theorem 1.1]{T.Kato} established small data GWP in some modulation spaces. Lu \cite[Theorems 6 and 7]{LuBnls} proved LWP for fourth-order \eqref{FNLSP} in $M_{s}^{p,2} \ ( 10/3 \leq p \leq 6, s > 7/2 -2/p
)$.
\item For well-posedness theory in the realm of generalized modulation spaces, so called $\alpha-$modulation space, we  refer to   \cite{alphaMS3} and the references therein. We  also note that Chaichenets and Pattakos \cite{NLSHOAID21,NLSHOAID23} established  global solution to the NLS with higher-order anisotropic dispersion for small initial data in  certain modulation spaces $M_{s}^{p,q}$.

\end{enumerate}
Taking the aforementioned results into account, we note that GWP results for \eqref{FNLSP} and \eqref{FNLSH}  \ $(\beta\neq 2)$ having large initial data in modulation space $M^{p,p'} \ (p>2)$ remain unsolved.  Theorems \ref{gwp} and \ref{gwpHar} resolve this for large radial initial data and complement well-posedness theory studied in Sobolev spaces.  

\subsection{Proof techniques and novelties}
The main ingredients in well-posedness theory  developed in several papers for   \eqref{FNLSP} in  modulation spaces are  the following:
\begin{itemize}
    \item The Schr\"odinger propagator $e^{-it(-\Delta)^{\beta/2}} $ is bounded on \( M^{p,q} \) and further  enjoys time decay estimates.  See Proposition \ref{uf} and cf. \cite{FabioPAMS}.
    \item The space $M^{p,1}_{s}$ is an algebra under pointwise multiplication. See e.g. \cite{Kasso2009,FeiStudia}.
    \item In \cite{RSJFA,alphaMS3,LuBnls},  authors used $U^{p}$-$V^{p}$ spaces taking values in modulation spaces as iteration spaces to establish LWP in $M^{2,q}_{s}$ for certain values of $s$ and $q.$ 
\end{itemize}
These properties  help to get the LWP and small data GWP in modulation spaces, see e.g. \cite{Kasso2009, wango, Wang2006, Bhimani2016,RSJFA, NLSHOAID21,NLSHOAID23}. 
However, there is no useful conservation law in modulation spaces. This raises  mathematical challenges  to establish GWP for large data. In this paper, we aim to address this for \eqref{FNLSP} and \eqref{FNLSH}.

Our method of proof is inspired from Bourgain’s high-low frequency decomposition method.
 It was introduced by Bourgain \cite{Bourgain1999} to establish GWP for NLS having power-type nonlinearity in $H^{s}$ for $s>\frac{3}{5}$. 
See \cite[Section 3.2]{KenigonBourgain}, \cite{vargas2001global} and \cite[Section 3.9]{TaoBook} for more details. The main idea is to break down the initial  data into a low-frequency/smoother component and
a high-frequency/rougher component. The choice of the cutoff for high and low-frequency component depends on the regularity of the data. The low-frequency component has finite energy after the cutoff, hence its evolution exists globally. On the other hand,  the nonlinear evolution to the difference
equation for the high-frequency component has small energy in an interval of time of existence. Such smallness of the Duhamel term in the smoother energy
space allows the iteration to continue by merging this smoother component with the evolution of the low-frequency component. Chaichenets et al. \cite{LeonidIn,leonidthesis} and Manna \cite{RameshManna} have  successfully adapted this method  for classical NLS \eqref{FNLSP}  $(\beta = 2)$  and classical Hartree eqaution \eqref{FNLSH} in  modulation spaces  respectively. In this paper, we extend this work to the case $ \frac{2n}{2n-1} < \beta < 2$ for $n \geq 2.$ We note that due to the non-locality and non-smoothness of the fractional dispersion operator $(-\Delta)^\beta$, the problem becomes significantly more challenging. Therefore, our analysis involves a new technicalities  and  we need to make  a radial assumption on the initial data. 

We briefly outline the strategy for our proof of Theorem \ref{gwp}.
   \begin{enumerate}
       \item[-] We decompose the initial data into two parts by making use of the interpolation result formulated in Lemma \ref{ipt} for radial functions as:
       \[u_0=\phi_0 + \psi_0 \in L^2_{rad} + M^{(\alpha+2), (\alpha +2)'}_{rad}\]
       such that  $\|\phi_0\|_{L^2} \lesssim_{p} N^{\gamma}$ for some $\gamma>0$ and $\|\psi_0\|_{M^{(\alpha+2), (\alpha+2)'}} \lesssim_{p} N^{-1}.$ See \eqref{dp} and \eqref{asi}.
\item[-] We have a local solution to \eqref{FNLS} in  $L^2_{rad} + M^{(\alpha+2), (\alpha +2)'}_{rad}$ by Theorem \ref{lwp}. Also, we have a global existence of \eqref{FNLS} for initial data in $L^{2}_{rad}$ due to mass conservation \eqref{mass}, see Theorem \ref{GWPL2Rad}.
       \item[-] To address global well-posedness, we construct a solution  $u$ of \eqref{FNLS} of the form \[u= (v_{0} + w_{0}) + e^{it\Delta} \psi_0.\] 
       Here $v_{0} $ is the $L^{2}_{rad}-$ global solution of \eqref{FNLS} with data $\phi_0$, see \eqref{gwpl2}. While $U_{\beta}(t)\psi_0 \in M^{(\alpha +2), (\alpha+2)'}$ for $t\in \R$ is the linear evolution of $\psi_0$ and $w_{0}$ is the nonlinear interaction term; see \eqref{w01}. See \eqref{solnlocal}.
       \item[-] Since $\|\psi_0\|_{M^{(\alpha+2), (\alpha+2)'}}\lesssim_{p} N^{-1}$ can be made small, we get $v_{0} + w_{0}$ close to $v_{0}$ in $L^2_{rad}.$ Even though $M(v_{0} + w_{0})$ is not conserved, it grows slowly enough to ensure the existence of a global solution. See Subsection \ref{secgwp} for the proof of Theorem \ref{gwp}.
   \end{enumerate} 

\begin{Remark} Overall proof strategy  for Theorem \ref{gwpHar} is similar to that of Theorem \ref{gwp}. However, we note that the technicalities are quite different due to the presence of non-local nonlinearity $(|x|^{-\nu}\ast |u|^2)u$.
\end{Remark}
\begin{Remark}
    The sign $+1$ (resp. $ -1$) corresponds to the focusing (resp. defocusing) case. However, the sign is irrelevant in our analysis.
\end{Remark}
\begin{Remark}

The proof of Theorem \ref{lwp}-\ref{gwpHar} relies on the radial
assumption, which removes the problematic evolution seen in the non-radial cases (such as the Knapp counterexample). The Strichartz estimates as stated in Theorem \ref{fst} without loss of derivatives hold only in the radial case. Guo et al. \cite{guo2014improved}
presented the Knapp counterexample to show that the general non-radial Strichartz
estimates have loss of derivative for $1 <\beta < 2.$ Moreover, with these results, we may not have GWP in $L^2$ with the desired property (cf. Theorem \ref{GWPL2Rad} and \cite[Proposition 3.4]{DGBJDE}) which is essential to proving Theorem \ref{gwp} and Theorem \ref{gwpHar} respectively.
   On the other hand, it would be interesting to  investigate \eqref{FNLS} for the  remaining dispersion  $\beta$  and non-radial initial data.  
\end{Remark}

This article is organized as follows: In Section \ref{np}, we introduce the basic notations and tools that will be used throughout the paper. This section includes definitions of function spaces, some of the key properties and relevant lemmas. Section \ref{s3} is dedicated to proving a global well-posedness result for \eqref{FNLSP} in $L^2_{rad}$ which will serve as a fundamental tool for obtaining global well-posedness results in the context of modulation spaces (Theorem \ref{gwp}). While Section \ref{s4} presents the proofs of Theorems \ref{lwp} and \ref{gwp}. Finally, in Section \ref{s5}, we discuss the proofs of Theorems \ref{lwpHar} and \ref{gwpHar}.

\section{Preliminaries}\label{np}
  \subsection{Notations}\noindent The symbol $X \lesssim_{\alpha} Y$ means 
 $X \leq CY$ 
for some positive constant $C$ depending on the parameter $\alpha.$ While $X \approx Y $ means $C^{-1}X\leq Y \leq CX$ for some constant $C>0.$ 
 The norm of the space-time Lebesgue space $L^{q}([0,T],L^{r}(\R^{n}))$ is defined as
$$\|u\|_{L^{q}_{T}L^{r}}:=\|u\|_{L^{q}([0,T],L^{r}(\R^{n}))}=\left(\int_{0}^{T} \|u(\cdot,t)\|_{L^{r}(\R^{n})}^{q} dt\right)^{\frac{1}{q}}.$$
We simply write $\|u\|_{L^{q}L^{r}}$ in place of $\|u\|_{L^{q}(\R,L^{r}(\R^{n}))}.$ 
\subsection{Key Ingredients} 
For $f\in \mathcal{S}(\mathbb R^{n})$ and $\beta >0,$ we define the \textbf{fractional Schr\"odinger propagator $e^{-it(-\Delta)^{\beta/2}}$ }for $t \in \mathbb R$ as follows:
\begin{eqnarray}
\label{sg}\label{-1}
[U_{\beta}(t)f](x)=\left[e^{-it (-\Delta)^{\beta/2}}f\right](x):= \int_{\mathbb R^n}  e^{-(2\pi)^\beta i |\xi|^{\beta}t}\, \widehat{f}(\xi) \, e^{2\pi i \xi \cdot x} \, d\xi.
\end{eqnarray}

\begin{Definition} A pair $(q,r)$ is \textbf{$\beta$-fractional admissible pair} if  $q\geq 2, r\geq 2$ and
$$\frac{\beta}{q} =  n \left( \frac{1}{2} - \frac{1}{r} \right),(q,r,n)\neq(\infty,2,2).$$
\end{Definition}

The set of all such admissible pairs is denoted by $$\mathcal{A}_{\beta}= \{(q,r):(q,r) \; \text{is $\beta$-fractional admissible pair} \}.$$

\begin{prop}[Strichartz estimates for fractional Schr\"odinger equation] \label{fst}Denote
$$DF(x,t):=  U_{\beta}(t)u_{0}(x)  + i\int_0^t U_\beta(t-s)F(x,s) ds.$$
\begin{enumerate}
\item \label{fst1}  \cite{KeelTao1998} Let $u_{0} \in L^2,$ $n\in \mathbb N$ and $\beta=2.$   Then for any time interval $I\ni0$ and 2-admissible pairs $(q_j,r_j)$, $j=1,2,$ 
satisfying 
$$ \|D(F)\|_{L^{q_1}(I,L^{r_1})}  \lesssim_{n,r_1,r_2}   \|u_{0} \|_{L^2}+    \|F\|_{L^{q'_2}(I,L^{r'_2})}, \quad\forall F \in L^{q_2'} (I, L^{r_2'}(\R^n))$$  where $q_j'$ and $ r_j'$ are H\"older conjugates of $q_j$ and $r_j$
respectively.

\item \label{fst2} \cite[Corollary 3.10]{guo2014improved} Let  $u_{0} \in L_{rad}^2$, $n\geq 2,$ and  $\frac{2n}{2n-1} < \beta < 2.$   Then for any time interval $I\ni0$ and $(q_j, r_j) \in \mathcal{A}_{\beta}$, $j=1,2,$ 
such that
$$ \|D(F)\|_{L^{q_1}(I,L^{r_1})}  \lesssim_{n,\beta,r_1, r_2}   \|u_{0}\|_{L^2}+     \|F\|_{L^{q'_2}(I,L^{r'_2})}, \quad \forall F \in L^{q'_2}(I,L_{rad}^{r'_2}(\R^n))$$  where $q_j'$ and $ r_j'$ are H\"older conjugates of $q_j$ and $r_j$ respectively.
\end{enumerate}
\end{prop}
 
\begin{lem}[see e.g. Lemma 3.9 in {\cite{leonidthesis}} ] \label{eg} Denote  \begin{equation*} G(u,v,w)=|u+v|^{\alpha}(u+v)-|u+w|^{\alpha}(u+w)
\end{equation*}
Then, for  $\alpha>0$ and $u,v,w\in \C$, we have 
     $$|G(u,v,w)| \lesssim_{\alpha} (|u|^{\alpha}+|v|^{\alpha}+|w|^{\alpha})|v-w|.
    $$
  \end{lem}
 \begin{lem}[Hardy-Littlewood Sobolev inequality]\label{HLS} Assume that $0<\nu <n$ and $1<p<q<\infty$ with $$\frac{1}{p}+\frac{\nu}{n}-1=\frac{1}{q}.$$
Then the map $f \to |\cdot|^{-\nu} \ast f$ is bounded from $L^p(\R^n )$ to $L^q(\R^n )$ as
$$ \||\cdot|^{-\nu} \ast f\|_{L^q} \lesssim_{n,\nu,p} \|f\|_{L^p}.$$
\end{lem}
Consider the identity for $\nu>0:$
\begin{equation}\label{haridentity}
(|\cdot|^{-\nu} \ast |u_1|^{2})u_1 - (|\cdot|^{-\nu} \ast |u_2|^{2})u_2 = (|\cdot|^{-\nu} \ast |u_1|^{2})(u_1 - u_2) + (|\cdot|^{-\nu} \ast (|u_1|^2 - |u_2|^2))u_2.
\end{equation}
Define 
\begin{equation}\label{TildeG}
    \tilde{G}(v,w_{1},w_{2})=(|\cdot|^{-\nu}\ast|v+w_{1}|^{2})(v+w_{1})-(|\cdot|^{-\nu}\ast|v+w_{2}|^{2})(v+w_{2})
\end{equation}
\begin{lem}[Basic property, \cite{KassoBook,wang2011harmonic, kobayashi2011inclusion}] \label{srp}
    Let $p,q,p_{i},q_{i} \in [1,\infty], (i=0,1,2)$ and $s_1, s_2 \in \R.$ 
  \begin{enumerate}
     \item $ \label{srp1}  M_{s_1}^{p_{1},q_{1}}\hookrightarrow M_{s_2}^{p_{2},q_{2}}$ whenever $p_{1}\leq p_{2}, q_{1}\leq q_{2},  s_2 \leq s_1.$
    \item \label{srp2} $  M_{s_1}^{p,q_{1}}\hookrightarrow M_{s_2}^{p,q_{2}}$ whenever $ q_{2}<q_{1},  s_1-s_2 > \frac{n}{q_2}-\frac{n}{q_1}$. 
            \item\label{embedding} $M^{p,q_1}\hookrightarrow L^{p} \hookrightarrow  M^{p,q_2} $ for $q_1 \leq \min \{p, p'\}$ and $q_2 \geq \max \{p, p'\}.$
            
\item \label{srpa}If $\frac{1}{p_{1}}+\frac{1}{p_{2}}=\frac{1}{p_{0}}$ and $\frac{1}{q_{1}}+\frac{1}{q_{2}}=1+\frac{1}{q_{0}},$ then $\|f g\|_ {M^{p_{0},q_{0}}} \lesssim  \|f\|_ {M^{p_{1},q_{1}}} \|g\|_ {M^{p_{2},q_{2}}}.$
     \item Denote 
$$\tau (p,q)= \max \left\{ 0, n\left( \frac{1}{q}- \frac{1}{p}\right), n\left( \frac{1}{q}+ \frac{1}{p}-1\right) \right\}.$$Then 
$L^{p}_{s_1} \subset M^{p,q}_{s_2}$ if and only if one of the following conditions is satisfied:
\begin{gather*}
 (i) \  q\geq p>1, s_1\geq s_2 + \tau(p,q); \quad (ii) \ p>q, s_1>s_2+ \tau(p,q);\\
(iii) \ p=1, q=\infty, s_1\geq s_2 + \tau(1, \infty); \quad (iv) \ p=1, q\neq \infty, s_1>s_2+\tau (1, q).
\end{gather*}
\item For completely sharp embedding between $M^{p,q}_{s_1}$ and $H^s,$ see  \cite{KassoBook}.
    \end{enumerate}
    \end{lem}

\begin{prop}[Boundedness of $U_{\beta}(t),$ \cite{chendcds,wango}]\label{uf}
The following inequalities hold:
 \begin{enumerate}
 \item For $\frac{1}{2} < \beta \leq 2$ and $ 1 \leq p,q\leq\infty,$ we have 
 $$ \|U_{\beta}(t)f \|_{M^{p,q}}\leq  (1+|t|)^{n\left| \frac{1}{p}-\frac{1}{2} \right|} \|f\|_{M^{p,q}}.$$
 \item For $\beta \geq 2 , 1\leq  q\leq \infty$ and $ 2 \leq  p\leq \infty,$ we have $$ \|U_{\beta}(t)f \|_{M^{p,q}}\leq  (1+|t|)^{- \frac{2n}{\beta}\left( \frac{1}{2}-\frac{1}{p} \right)} \|f\|_{M^{p',q}}.$$ 
 \end{enumerate} 
\end{prop}
By invoking interpolation \cite[Theorem 6.1 (D)]{Feih83} theory, we may split the initial data, with the desire  property.  More specifically, we have the following lemma, which will play a crucial role to prove our main results.
\begin{lem}[see Proposition 5.1 in {\cite{leonidthesis}}]\label{ipt} Let $r>2$ and $ \ p\in (2, r).$ Then 
 there exists a constant $C=C(n,r,p)$ such that  for  any $u\in M^{p,p'}$ and  $N>0$  there are 
$\ v \in L^{2}$ and $w \in M^{r,r'}$ satisfying
\begin{gather*}
\begin{cases}
u= v + w\\
 \|v\|_{L^{2}}\leq C \|u\|
        _{M^{p,p'} }N^{\gamma}\\ 
        \|w\|_{M^{r,r'}}\leq C \|u\|_{M^{p,p'}}/N 
\end{cases}, \ \ where  \   \gamma = \frac{\frac{1}{2} - \frac{1}{p}}{\frac{1}{p} - \frac{1}{r}}.
\end{gather*} \end{lem}

\section{Global well-posedness in $L^{2}_{rad}$}\label{s3}
In this section, we recall \cite[Corollary 4.8]{guo2014improved} in Theorem \ref{GWPL2Rad} below for the sake of completeness. This result will be crucial in establishing local and global well-posedness in modulation spaces. \\
Denote
\begin{align}
\label{exponent1}
    \omega&:=1-\frac{n\alpha}{2\beta},\\
\label{kappa}
     \kappa&:=\frac{n\alpha}{2\beta(\alpha+2)},\\
\label{YT}
 Y(T) &:= L^{1 / \kappa}_{T}L^{\alpha+2}_{rad}.
\end{align}
\begin{lem}\label{lemlwp}Let $n \geq 2, \alpha \in (0,\frac{2\beta}{n}), \beta \in (\frac{2n}{2n-1}, 2)$ and $(q_{1},r_{1})\in \mathcal{A}_{\beta}$. Then, there exists a constant $C=C(n,\alpha,\beta,r_{1})>0$
such that for any $T > 0$  
and any $u,v,w \in Y(T)$, the estimate
\begin{align*}
   \hspace{-0.5cm}\left|\left| \int_0^t U_\beta(t-s)G(u,v,w)(s) ds \right|\right|_{L^{q_{1}}_{T}L^{r_{1}}} \lesssim_{n,\alpha,\beta,r_{1}} T^{\omega}\|v-w\|_{Y(T)}\left(\|u\|_{Y(T)}^{\alpha}
     +\|v\|_{Y(T)}^{\alpha} + \|w\|_{Y(T)}^{\alpha}\right)
\end{align*}
holds.
\end{lem}
\begin{proof}
  Using Proposition \ref{fst}\eqref{fst2}, we have 
  \begin{align*}
   \left|\left| \int_0^t U_\beta(t-s)G(u,v,w)(s) ds \right|\right|_{L^{q_{1}}_{T}L^{r_{1}}} \lesssim_{n,\beta,r_{1},r_{2}}  \left|\left| G(u,v,w) \right|\right|_{L^{q'_{2}}_{T}L^{r'_{2}}} 
\end{align*}
for $(q_j, r_j) \in \mathcal{A}_{\beta}$, $j=1,2.$ 
Applying Lemma \ref{eg} and H\"older's inequality twice with $(q_{2},r_{2})=(1 / \kappa,\alpha+2),$  we have\begin{align*}
    \left|\left| G(u,v,w) \right|\right|_{L^{q'_{2}}_{T}L^{r'_{2}}} 
    \lesssim_{\alpha}&\|u^{\alpha}(v-w)\|_{L^{q'_{2}}_{T}L^{r'_{2}}} + \|v^{\alpha}(v-w)\|_{L^{q'_{2}}_{T}L^{r'_{2}}} + \|w^{\alpha}(v-w)\|_{L^{q'_{2}}_{T}L^{r'_{2}}}
    \end{align*}
    \begin{align*}
    \leq T^{\omega}\|v-w\|_{L^{1 / \kappa}_{T}L^{\alpha+2}}\left( \|u\|_{L^{1 / \kappa}_{T}L^{\alpha+2}}^{\alpha} +\|v\|_{L^{1 / \kappa}_{T}L^{\alpha+2}}^{\alpha}+ \|w\|_{L^{1 / \kappa}_{T}L^{\alpha+2}}^{\alpha}\right).
\end{align*}
\end{proof}

\begin{thm}[Global well-posedness in $L^{2}_{rad}$]\label{GWPL2Rad}Let $n \geq 2,  0 < \alpha < \frac{2\beta}{n}$ and $ \frac{2n}{2n-1} < \beta < 2.$ If $u_{0}\in L^{2}_{rad}$, then \eqref{FNLS} has a unique global solution
$$u\in C(\R,L^2_{rad}) ~\cap~ L^{1 / \kappa}_{loc}(\R,L^{\alpha+2}).$$
Moreover, $$\|u(t)\|_{L^{2}}=\|u_{0}\|_{L^{2}} \; \forall  t\in \R $$
and for all $(q,r) \in \mathcal{A}_{\beta},~u \in L^{q}_{loc}(\R,L^{r}).$
\end{thm}
\begin{proof}
By Duhamel's principle,  \eqref{FNLSP} is formulated as the integral equation
\begin{eqnarray}\label{DP}
  u(x,t)=U_{\beta}(t)u_{0}(x) \pm i  \int_0^t U_\beta(t-s)G(0,u,0)(s) ds :=\Lambda (u)(x,t).  
\end{eqnarray}
    Let $R$ and $T$ be a positive real numbers to be chosen later. Define $$B(R,T)=\{u\in X_{1}(T):\|u\|_{X_{1}(T)}\leq R\}$$ where
    $$X_1(T)=C_{T}L^2_{rad} ~\cap~ L^{1 / \kappa}_{T}L^{\alpha+2}$$
equipped with the norm 
$$\|v\|_{X_1(T)}=\|v\|_{L^{\infty}_{T}L^2} + \|v\|_{L^{1 / \kappa}_{T}L^{\alpha+2}}.$$
 Using Proposition \ref{fst}\eqref{fst2} with $r_{1} \in \{2,\alpha+2\},$ we have 
\begin{equation}\label{RBound}
    \|U_{\beta}(t)u_{0}\|_{X_1(T)} \lesssim_{n,\alpha,\beta} \|u_{0}\|_{L^2}.
\end{equation}
This suggests the choice of $$R=2C(n,\alpha,\beta)\|u_{0}\|_{L^2}.$$
Using Lemma \ref{lemlwp} with $r_{1} \in \{2,\alpha+2\},$ we have
\begin{align}
   \left|\left| \int_0^t U_\beta(t-s)G(0,u,0)(s) ds \right|\right|_{X_{1}(T)} \nonumber
   &\lesssim_{n,\alpha,\beta} T^{\omega} \|u\|_{Y(T)}^{\alpha+1}\nonumber\\
   & \label{TBound}\leq T^{\omega} R^{\alpha+1} \leq R,
\end{align}
provided $u\in B(R,T)$ and $$T \lesssim_{n,\alpha,\beta} \|u_{0}\|_{L^{2}}^{-\alpha / \omega}.$$
This shows that $\Lambda(u) \in B(R,T)$. Similarly, we can show that $\Lambda(u)$ is a contraction map from $B(R,T)$ to $B(R,T)$. Using Banach contraction principle, there exists a unique solution $u$ to  \eqref{FNLSP} in the time interval of existence $[0,T]$ with $T$ depending on $\|u_{0}\|_{L^{2}}$. Since  \eqref{FNLSP} enjoys conservation of mass \eqref{mass}, we can extend the solution globally.
\par{The last property follows from the Strichartz estimates applied with an arbitrary $\beta$- fractional admissible pair on the left hand side and same pair on the right hand side of \eqref{RBound} and \eqref{TBound}. }
\end{proof}
\section{Proof of  Theorems \ref{lwp} and \ref{gwp}}\label{s4}
\subsection{Local well-posedness in $L^2_{rad}+M^{\alpha+2,(\alpha+2)'}_{rad}$}\label{seclwp}In this subsection, we shall prove Theorem \ref{lwp}. First, we introduce some notation before proceeding with the main proof.

We are going to work in the Banach space $X(T)$ expressed as
\begin{equation}
\begin{aligned}\label{X(T)}
        X(T)&:= X_{1}(T)+X_{2}(T)
\end{aligned}
\end{equation}
where $$X_1(T):=C_{T}L^2_{rad} ~\cap~ L^{1 / \kappa}_{T}L^{\alpha+2}$$
equipped with the norm 
$$\|v\|_{X_1(T)}=\max \left\{\|v\|_{L^{\infty}_{T}L^2},\|v\|_{L^{1 / \kappa}_{T}L^{\alpha+2}}\right\} $$ and $$ X_2(T):=C_{T}M^{\alpha+2,(\alpha+2)'}_{rad}.$$\\
The norm on $X(T)$ is given as 
\begin{equation*}
\|u\|_{X(T)} =\inf_{\substack{u=v+w \\ v \in X_1(T) \\ w \in X_2(T)}} \left(\|v\|_{X_1(T)} + \|w\|_{X_2(T)} \right).
\end{equation*}
\begin{Remark}\label{lemlwp2}
    For $T\leq 1$, by H\"older's inequality for the time variable and Lemma \ref{srp}\eqref{embedding},  we have $\|\cdot\|_{Y(T)} \lesssim_{n} \|\cdot\|_{X(T)}.$
\end{Remark}

\begin{proof}[\textbf{Proof of Theorem \ref{lwp}}] 
Let $a$ and $T$ be a positive real numbers (to be chosen later).  Define $$B(a,T)=\{u\in X(T):\|u\|_{X(T)}\leq a\}.$$
We shall show that $\Lambda(u)$ (as defined in \eqref{DP}) is a contraction mapping on $B(a, T).$
Assume w.l.o.g. $T\leq1.$ Consider an arbitrary decomposition of $u_{0}=v_{0}+w_{0}\in L^{2}_{rad}+ M^{\alpha+2,(\alpha+2)'}_{rad}, v_{0}\in L^{2}_{rad}$ and $w_{0}\in M^{\alpha+2,(\alpha+2)'}_{rad}.$
For the linear evolution of $u_{0},$ using Proposition \ref{fst}\eqref{fst2} and Proposition \ref{uf}, we have 
\begin{align*}
     \|U_{\beta}(t)u_{0} \|_{X(T)}&\leq  \|U_{\beta}(t)v_{0} \|_{X_{1}(T)}+\|U_{\beta}(t)w_{0} \|_{X_{2}(T)} \\ &\lesssim_{n,\alpha,\beta} \|v_{0}\|_{L^{2}}+(1+T)^{n\left( \frac{1}{2}-\frac{1}{\alpha+2} \right)} \|w_{0}\|_{M^{\alpha+2,(\alpha+2)'}}\\
&\lesssim_{n}\|v_{0}\|_{L^{2}}+\|w_{0}\|_{M^{\alpha+2,(\alpha+2)'}}.
\end{align*}
Since the decomposition is arbitrary, it follows that
    $$\|U_{\beta}(t)u_{0} \|_{X(T)} \lesssim_{n,\alpha,\beta} \|u_{0}\|_{L^2+M^{\alpha+2,(\alpha+2)'}}.$$
This suggests the choice of $a=2C(n,\alpha,\beta)\|u_{0}\|_{L^{2}+M^{\alpha+2,(\alpha+2)'}}.$ The integral part is estimated using $X_{1}(T)\hookrightarrow X(T),$ Lemma \ref{lemlwp} (with $r_{1} \in \{\alpha+2,2\}$) and Remark \ref{lemlwp2} as follows
    \begin{align}
    \left|\left| \int_{0}^{t} U_{\beta}(t-\tau)(|u|^{\alpha}u)(\tau)d\tau \right|\right|_{X(T)}
    &\lesssim \left|\left| \int_{0}^{t} U_{\beta}(t-\tau)(|u|^{\alpha}u)(\tau)d\tau \right|\right|_{X_{1}(T)} \nonumber\\
    &\lesssim_{n,\alpha,\beta} T^{\omega} \|u\|_{Y(T)}^{\alpha+1}\nonumber\\
    &\lesssim_{n} T^{\omega} \|u\|_{X(T)}^{\alpha+1} \leq T^{\omega}a^{\alpha+1} \nonumber
    \end{align}
provided $u\in B(a,T)$. Taking 
    \begin{equation}\label{Tchoice}
        T :=\min \left\{ 1,~C(n,\alpha,\beta)  \|u_{0}\|_{L^2+M^{\alpha+2,(\alpha+2)'}}^{-\alpha /\omega } \right\}~,
    \end{equation}
    we conclude $\Lambda(u)\in B(a, T).$ Similarly,  one can show that $\Lambda(u)$ is a contraction mapping. Using Banach contraction mapping principle, we obtain a unique fixed point $u$ for $\Lambda,$ which is a solution to \eqref{FNLS} on the time interval $[0,T]$.
\end{proof}
\subsection{Global well-posedness in $M^{p,p'}_{rad}$}\label{secgwp} In this subsection, we shall prove that, when 
$u_{0} \in M^{p,p'}_{rad},$ the local solution  $u$ obtained in Theorem \ref{lwp} can be extended globally in time.
Suppose to the contrary that the solution established in 
Theorem \ref{lwp} is not global in time. Thus, we have the maximal time $T^*< \infty.$  In this case,  we shall  produce a solution $u$  of \eqref{FNLS} (to be defined in \eqref{ss} below), which will exist on a larger interval $[0, T_1]$ for $T_1> T^*.$ This will lead to a contradiction to the maximal time interval $[0, T^*).$

We start by decomposing initial data $u_0 \in M^{p, p'}_{rad} \subset L^2_{rad} + M^{(\alpha +2), (\alpha +2)'}_{rad} $ into two parts such that the size of $M^{(\alpha +2), (\alpha +2)'}-$data can be controlled by  arbitrary small quantity.\\  Using Lemma \ref{ipt} for radial functions,  for any $N>1$ and $u_0 \in M^{p, p'}_{rad},$ there exists  $\phi_0 \in L^2_{rad}$ and $ \psi_0 \in M^{(\alpha +2), (\alpha +2)'}_{rad}$ 
such that 
\begin{equation}\label{dp}
    u_0= \phi_0 + \psi_0
\end{equation}
with 
\begin{eqnarray}\label{asi}
\|\phi_0\|_{L^2} \lesssim N^{\gamma},  \quad \|\psi_0\|_{M^{(\alpha +2), (\alpha +2)'}} \lesssim  1/N
\end{eqnarray}
where\begin{equation}\label{betap}
    \gamma = \frac{\frac{1}{2} - \frac{1}{p}}{\frac{1}{p} - \frac{1}{\alpha+2}}.
\end{equation}
Now, Consider  \eqref{FNLSP} having initial data $\phi_0,$ namely
\begin{eqnarray}
\begin{cases}
 i \partial_t v_{0} + (-\Delta)^\frac{\beta}{2} v_{0}\pm |v_{0}|^{\alpha}v_{0}=0 \\
v_{0}(\cdot,0)=\phi_{0}\in L^2_{rad}.   
\end{cases}        
    \end{eqnarray}
By Theorem \ref{GWPL2Rad},  \eqref{FNLSP} has a unique solution $v_{0}$ in the space
\begin{eqnarray}\label{gwpl2}
  C(\R, L^2_{rad}) \cap L_{loc}^{q}(\R,L^{r} )
\end{eqnarray} with $(q,r) \in \mathcal{A}_{\beta}$ and satisfies
\begin{equation}\label{v0qr2}
\sup_{(q,r)\in \mathcal{A}_{\beta}}\|v_{0}\|_{L_{loc}^{q} L^{r}} \lesssim_{n,r,\beta} \|\phi_{0}\|_{L^{2}}.
\end{equation}
Next, we consider the modified  \eqref{FNLSP}  corresponding to the evolution of $\psi_{0}$: 
\begin{eqnarray}\label{ivpMod}
\begin{cases}
 i \partial_t w + (-\Delta)^\frac{\beta}{2} w \pm (|w+v_{0}|^{\alpha}(w+v_{0})-|v_{0}|^{\alpha}v_{0})=0 \\
w(\cdot,0)=\psi_{0}\in M^{\alpha+2, (\alpha+2)'}_{rad} .
\end{cases}
\end{eqnarray}
The solution to the above I.V.P. \eqref{ivpMod} is given as
\begin{equation}
U_{\beta}(t)\psi_{0}+w_{0},
\end{equation}
where $w_0$ is the nonlinear interaction associated with $\psi_0,$ expressed as \begin{align}
    w_{0} &= \pm i \int_{0}^{t}  U_\beta(t-\tau)\left( \big| U_{\beta}(\tau)\psi_{0}+w_{0}+v_{0}\big|^{\alpha} (U_{\beta}(\tau)\psi_{0}+w_{0}+v_{0}) - |v_{0}|^{\alpha} v_{0}\right) \, d\tau \nonumber\\
    &=\pm i \int_{0}^{t} U_\beta(t-\tau)G(v_{0} + U_{\beta}(\tau)\psi_{0}, w_{0},0) \, d\tau  \pm i \int_{0}^{t} U_\beta(t-\tau)G(v_{0}, U_{\beta}(\tau)\psi_{0},0)\, d\tau.\label{w01}
\end{align}
Thus, the solution to the  \eqref{FNLSP}  having initial data $u_{0}$ in \eqref{dp} can be written as follows
\begin{eqnarray}\label{solnlocal}
v_{0}+U_{\beta}(t)\psi_{0}+w_{0}.
\end{eqnarray}
 
 We remark that  $v_0$ and $ U_{\beta}(t)\psi_0$ are globally defined in appropriate spaces as a result of  \eqref{gwpl2} and Proposition \ref{uf}. To establish the desired global existence, we first examine the time interval of existence for $w_0$. To this end,  we prove local existence of solution to the integral equation \eqref{w01} for general $\phi$ and $\psi$ as it is required at each stage of the iteration.
\begin{prop}\label{w0exist} Let $\phi \in L^2_{rad}, \psi\in M^{(\alpha +2), (\alpha +2)'}_{rad}$ and $Y(T)$ be as in  \eqref{YT}.  Assume that  $\beta$ and $\alpha$  be  as in Theorem \ref{gwp}. Denote by  $v$  the $L^2-$global solution for initial value $\phi$.
Then  there exists  a constant $C=C(n, \alpha,\beta)>0$ such that the integral equation $$w= \pm i \int_{0}^{t} U_\beta(t-\tau) G(v + U_{\beta}(\tau)\psi, w,0) \, d\tau   \pm i \int_{0}^{t}  U_\beta(t-\tau) G(v, U_{\beta}(\tau)\psi,0) \, d\tau$$ has a unique solution  $w \in Y(T)$   provided $T$ satisfying
\begin{align}
    \label{c1} T &\leq 1 \\
    \label{c2} T &\leq C \left( \|\phi\|_{L^{2}} + \|\psi\|_{M^{\alpha+2,(\alpha+2)'}} \right)^{-\alpha /\omega} \\
    \label{c3} T &\leq C \left( \|\psi\|_{M^{\alpha+2,(\alpha+2)'}} \right)^{-\alpha / (\omega + \alpha \kappa)}.
\end{align}
\end{prop}
\begin{proof}
     Define $$B(A,T)=\{w\in Y(T):\|w\|_{Y(T)}\leq A\}$$
   with $A>0$ (to be choosen later). Let $T$ be the minimum of the right-hand sides of the conditions \eqref{c1},  \eqref{c2} and \eqref{c3}.
    Define the operator
    $$\Gamma(w):= \pm i \int_{0}^{t} U_\beta(t-\tau) G(v + U_{\beta}(\tau)\psi, w,0) \, d\tau   \pm i \int_{0}^{t}  U_\beta(t-\tau) G(v, U_{\beta}(\tau)\psi,0) \, d\tau.$$
   Using Lemma \ref{lemlwp} and embedding $L^{\infty}_{T} \hookrightarrow L^{1 / \kappa}_{T}$ under the assumption \eqref{c1}, for any $v_{1},v_{2} \in \C,$ we have 
    \begin{align}
    \left\| \int_{0}^{t} U_\beta(t-\tau)G(v_{1},v_{2},0) \, d\tau \right\|_{Y(T)}
   \lesssim_{n,\alpha,\beta}\label{before4.10} & T^{\omega}
     \left(\|v_{1}\|^{\alpha}_{Y(T)} \|v_{2}\|_{Y(T)}+
     \|v_{2}\|^{\alpha + 1}_{Y(T)}\right)
     \end{align}\ignorespaces
     \begin{align}
     \label{4.10}\lesssim & \left( T^{\omega + \kappa} \right) 
     \|v_{1}\|^{\alpha}_{Y(T)} \|v_{2}\|_{L^{\infty}_{T}L^{\alpha+2}} 
   + \left( T^{\omega + \kappa(\alpha+1)} \right) 
     \|v_{2}\|^{\alpha + 1}_{L^{\infty}_{T}L^{\alpha+2}}.
\end{align}
Using the estimate \eqref{4.10} for $v_{1}=v , v_{2}=U_{\beta}(\tau)\psi$ along with \eqref{v0qr2}, Lemma \ref{srp} \eqref{embedding} and Proposition \ref{uf} under the assumption \eqref{c1}, we obtain\\
    \begin{align*} \hspace{-7cm}\left|\left| \displaystyle\int_{0}^{t} U_\beta(t-\tau)G(v,U_{\beta}(\tau)\psi,0) \, d\tau \right|\right|_{Y(T)} 
    \end{align*}
    \begin{align*}
   \lesssim_{\alpha,n,\beta}& \left( T^{\omega + \kappa} \right) \|v\|_{Y(T)}^{\alpha}  \| U_{\beta}(\tau)\psi \|_{L^{\infty}_{T}L^{\alpha+2}} + \left( T^{\omega + \kappa(\alpha+1)} \right) \| U_{\beta}(\tau)\psi \|_{L^{\infty}_{T}, 
     L^{\alpha+2}}^{\alpha+1} \\
     \lesssim_{\alpha,n,\beta} & \left( T^{\omega + \kappa} \right) \| \phi \|_{L^{2}}^{\alpha}  \| \psi\|_{M^{\alpha+2, (\alpha+2)'}}+\left( T^{\omega + \kappa(\alpha+1)} \right)  \| \psi \|_{M^{\alpha+2, (\alpha+2)'}}^{\alpha+1} \\
     =& T^{\kappa} \| \psi \|_{M^{\alpha+2, (\alpha+2)'}}\left( T^{\omega} \right. \| \phi \|_{L^{2}}^{\alpha} 
      +\left. T^{\omega + \alpha \kappa}\| \psi \|_{M^{\alpha+2, (\alpha+2)'}}^{\alpha}  \right) \\
     \lesssim_{\alpha,n,\beta} &   T^{\kappa} \| \psi \|_{M^{\alpha+2, (\alpha+2)'}}.
\end{align*}
The last inequality follows due to our assumptions \eqref{c2} and \eqref{c3}.
This suggests the choice of
    \begin{equation}\label{AA}
    A= \frac{3}{C(n,\alpha,\beta)}T^{\kappa}\|\psi\|_{M^{\alpha+2,(\alpha+2)'}}.
    \end{equation}
    where $C=C(n,\alpha,\beta)$ is the same constant that appears in \eqref{c2} and \eqref{c3}. Thus, \begin{equation}\label{gammaw1}
      \left|\left| \displaystyle\int_{0}^{t} U_\beta(t-\tau)G(v,U_{\beta}(\tau)\psi,0)\, d\tau \right|\right|_{Y(T)}\leq \frac{A}{3}.
    \end{equation}
     Using the estimate \eqref{before4.10} for $v_{1}=v + U_{\beta}(\tau)\psi$ and $ v_{2}=w$ along with \eqref{v0qr2}, Lemma \ref{srp} \eqref{embedding}  and Proposition \ref{uf} under the assumption \eqref{c1}, we have
\begin{align*}
    \hspace{-7cm} \left\| \int_{0}^{t}  U_\beta(t-\tau)G(v +  U_{\beta}(\tau)\psi, w,0) \, d\tau \right\|_{Y(T)} 
\end{align*}\ignorespaces
\begin{align*}
    & \lesssim_{\alpha,n,\beta} T^{\omega}   \left( \|v
    + U_{\beta}(\tau)\psi \|_{Y(T)}^{\alpha} \|w\|_{Y(T)} + \|w\|_{Y(T)}^{\alpha+1} \right) \\
    & \lesssim_{\alpha,n,\beta}  \|w\|_{Y(T)}\left\{T^{\omega}   \left( \left( \|\phi\|_{L^{2}} + \|\psi\|_{M^{\alpha+2, (\alpha+2)'}} \right)^{\alpha} +\|w\|_{Y(T)}^{\alpha} \right)\right\} \\
    & \lesssim_{\alpha,n,\beta}\|w\|_{Y(T)} \left\{\frac{1}{3} +T^{\omega+\alpha \kappa}\|\psi\|^{\alpha}_{M^{\alpha+2,(\alpha+2)'}}   \right\}.
\end{align*}
In the last inequality, we have used \eqref{c2} in the first summand and substitute the norm of $w$ in $Y(T)$ to the power of $\alpha$ by $A^{\alpha}$, ($A$ given in \eqref{AA}) in the second summand. Considering the second summand of last inequality under the assumption \eqref{c3}, we get
        \begin{equation}\label{gamma2}
            \left\| \int_{0}^{t} U_\beta(t-\tau)G(v + U_{\beta}(\tau)\psi, w,0)\, d\tau \right\|_{Y(T)} \leq \frac{2A}{3}.
        \end{equation}
        Combining \eqref{gammaw1} and \eqref{gamma2}, we can say that $\Gamma(w)$ belongs to $B(A, T).$ Contractivity of $\Gamma$ follows similarly \footnote{$C(n,\alpha,\beta)$ is choosen small enough such that all requirements are fulfilled.}. Thus, by the Banach fixed-point theorem, we get a unique fixed point $w$ to the integral equation \eqref{w01} on the time-interval $[0,T].$
\end{proof}  
\begin{Remark} \label{nexphar}
   Note that both of the exponents on right hand side of conditions \eqref{c2} and \eqref{c3} involving $L^{2}$ and  $M^{\alpha+2,(\alpha+2)'}-$ norms are  negative as $0<\alpha<\frac{2\beta}{n}$ .
\end{Remark}
 \begin{cor}\label{winfty2}
 Under the hypothesis of Proposition \ref{w0exist}, there exists a constant $C(n,\alpha,\beta)$ satisfying
    \begin{equation*}
\|w\|_{L^{\infty}_{T}L^{2}}\lesssim_{n,\alpha,\beta} T^{\kappa}\|\psi\|_{M^{\alpha+2,(\alpha+2)'}}.
    \end{equation*}
\end{cor}
\begin{proof}
Applying
Strichartz estimates (Proposition \ref{fst}) for $L^{\infty}_{T}L^{2}$ in place of $Y(T)-$ norm i.e. using an admissible pair $ (\infty,2)$ on the left hand side 
and the same pairs on the right hand side of \eqref{before4.10} and \eqref{4.10} in the proof of Proposition \ref{w0exist}, we obtain the desired result. Specifically,
\begin{align*}
\|w\|_{L^{\infty}_{T}L^{2}} \leq & \left|\left|\int_{0}^{t} U_\beta(t-\tau)G(v + U_{\beta}(\tau)\psi, w,0) \, d\tau \right|\right|_{L^{\infty}_{T}L^{2}}\\
&+\left|\left|\int_{0}^{t} U_\beta(t-\tau)G(v, U_{\beta}(\tau)\psi,0) \, d\tau \right|\right|_{L^{\infty}_{T}L^{2}}\\
&\lesssim_{n,\alpha,\beta} T^{\kappa}\|\psi\|_{M^{\alpha+2,(\alpha+2)'}}.
\end{align*}
\end{proof}

\begin{proof}[\textbf{Proof of Theorem \ref{gwp}}]
Denote the constant from Proposition \ref{w0exist} by $C=C(n,\alpha,\beta)$ and put
\begin{equation}\label{TN}
        T=T(N)=(3CN^{\gamma})^{-\alpha /\omega} .
    \end{equation}
Applying Proposition \ref{w0exist} for $\phi=\phi_{0}$ and $\psi=\psi_{0}$, we can say that the solution $u =v_0 + U_{\beta}(t) \psi_0 + w_0$ to  \eqref{FNLSP}  corresponding to the data $u_{0}$ in \eqref{dp}
exists in the interval $[0, T(N)]$.\\
We aim to extend our solution to the interval $[T(N), 2T(N)]$ by using similar procedure but with the new initial data as the sum of the following two functions:    
\[ \phi_{1}=v_{0}(T)+w_{0}(T) \quad \text{and} \quad \psi_{1} =U_{\beta}(T)\psi_{0}.\]
More generally, we 
 aim to  extend our solution further by using a similar procedure :
\begin{itemize}
    \item[--]We define $\phi_{k}$ and $\psi_k$ for $k \geq 1$ (for $k=0,\; \phi_{0}$ and $\psi_0$ are defined in \eqref{dp}) as follows:
    \begin{equation}\label{iter1}
 \phi_{k}= v_{k-1}(kT) + w_{k-1}(kT) \quad \text{and} \quad \psi_{k}= U_{\beta}(kT)\psi_{0}, 
\end{equation}
where
\begin{eqnarray}
    \begin{aligned}
     w_{k-1}(kT)=\pm i \displaystyle \int_{0}^{kT} U_{\beta}(kT-\tau) G(v_{k-1}+U_{\beta}(k\tau)\psi_{0},w_{k-1},0) d\tau \nonumber \\
    \pm i \displaystyle \int_{0}^{kT} U_{\beta}(kT-\tau)G(v_{k-1},U_{\beta}(k\tau)\psi_{0},0) d\tau.
    \end{aligned}
       \end{eqnarray}
       \item[--] Assume for $kT\leq T^* 
       \footnote{Note that $(K-1)T\leq T^*.$ Since $T\leq 1,$ we will have $KT\leq T^{*}+1.$}, \phi=\phi_{k}$ and $\psi=U_{\beta}(kT)\psi_{0}$, $T$ satisfy all three conditions \eqref{c1}, \eqref{c2} and \eqref{c3} of Proposition \ref{w0exist}, where $k \in \{0, 1,\cdots, K-1\}.$
       \item[--]Let $v_{k}$ be INLS evolution of  $\phi_k$, and by construction 
\begin{equation}\label{ss}
   u(\cdot, t)= v_{k} (\cdot, t-kT ) + w_{k}(\cdot, t-kT) + U_{\beta}(t) \psi_0,  \quad \text{if} \ t\in [kT, (k+1)T] 
\end{equation}
 defines a solution of \eqref{FNLS} for $k \in \{0, 1,\cdots, K-1\}$.  
 \end{itemize}
 We shall show that  the iterative process terminates with  $KT >T^*.$ Since  $v_{K}$ and $U_{\beta}(t)\psi_0$ are globally defined in appropriate spaces, we are left to extend the solution of nonlinear interaction term $w_{K}$ to the $K$th iteration. To do this, we shall use Proposition \ref{w0exist} with $\phi=\phi_{K}$ and $\psi=U_{\beta}(KT)\psi_{0}.$  

Considering Remark \ref{nexphar} and \eqref{TN}, $T=T(N)\to 0$ as $N\to \infty$. Thus,
the  smallness condition \eqref{c1} is satisfied independently of $k$ for large $N.$\\
 Using Proposition \ref{uf} and  \eqref{asi}, we have
\begin{align}\label{ref1} 
\|U_{\beta}(t)\psi_{0}\|_{L^{\infty}([0,T^{*}+1],M^{\alpha+2,(\alpha+2)'}) } &\lesssim_{n,T^*}\|\psi_{0}\|_{M^{\alpha+2,(\alpha+2)'}} \lesssim_{n} 1/N \xrightarrow{N \to \infty} 0.
    \end{align}
Inserting $U_{\beta}(kT)\psi_{0}$ in the right hand side of \eqref{c3}, we have
\begin{eqnarray*}
\begin{aligned}
    \left( \|U_{\beta}(kT)\psi_{0}\|_{M^{\alpha+2,(\alpha+2)'}} \right)^{-\alpha /(\omega + \alpha \kappa)}  \gtrsim_{n} N^{\alpha /(\omega + \alpha \kappa)} \xrightarrow{N \to \infty} \infty.
   \end{aligned}
\end{eqnarray*}  
Since the lower bound does not depend on $k$ and $T \xrightarrow{N \to \infty} 0,$  the condition \eqref{c3} holds for sufficiently large $N.$

\noindent 
Thus, we either have $KT> T^*$ or the condition \eqref{c2} fails in the last iterative step $k=K,$ i.e.
\begin{equation}\label{aim11}
    3CN^{\gamma}< \|\phi_{K}\|_{L^2}+ \|U_{\beta}(KT)\psi_{0}\|_{M^{\alpha+2,(\alpha+2)'}}.
\end{equation} 
Considering \eqref{ref1}, \eqref{aim11} can be written as
\begin{equation}\label{aim}
    3CN^{\gamma}< \|\phi_{K}\|_{L^2}+CN^{\gamma}.
\end{equation}
We claim that even under condition \eqref{aim},   we still have $KT>T^*$. This  clearly leads to a contradiction with the definition of $T^*.$ 

Considering the construction of $\phi_k$ and Corollary \ref{winfty2}, we note that $\phi_k \in L^2$ for $k\in \{0, 1,\cdots, K-1\}.$ Now exploiting the conservation of mass \eqref{mass} and Corollary \ref{winfty2} (for  $w=w_{k}$ and $\psi=U_{\beta}(kT)\psi_{0}, \; 0\leq k \leq K-1$), we have
    \begin{align}
        \|\phi_{K}\|_{L^{2}}  &\leq \|v_{K-1}\|_{L^{\infty}_{[(K-1)T,KT]}L^{2}}+\|w_{K-1}\|_{L^{\infty}_{[(K-1)T,KT]}L^{2}} \nonumber\\
        &= \|\phi_{K-1}\|_{L^2} + \|w_{K-1}\|_{L^{\infty}_{[(K-1)T,KT]}L^2}\nonumber\\
        & \leq \|v_{K-2}\|_{L^{\infty}_{[(K-2)T,(K-1)T]}L^{2}}+\|w_{K-2}\|_{L^{\infty}_{[(K-2)T,(K-1)T]}L^{2}} +\|w_{K-1}\|_{L^{\infty}_{[(K-1)T,KT]}L^{2}} \nonumber\\
        &= \|\phi_{K-2}\|_{L^2}+\|w_{K-2}\|_{L^{\infty}_{[(K-2)T,(K-1)T]}L^{2}} +\|w_{K-1}\|_{L^{\infty}_{[(K-1)T,KT]}L^{2}} \nonumber\\
        & \leq \cdots \leq  \|\phi_{0}\|_{L^{2}}+\sum_{k=0}^{K-1}\|w_{k}\|_{L^{\infty}_{[kT,(k+1)T]}L^{2}} \nonumber \\
        & \lesssim_{n, \alpha,\beta} N^{\gamma}+T^{\kappa}\sum_{k=0}^{K-1}\|U_{\beta}(kT)\psi_{0}\|_{M^{\alpha+2,(\alpha+2)'}}\nonumber\\
        &\lesssim_{n,T^{*}} N^{\gamma}+T^{\kappa}K\frac{1}{N}.\label{long}
    \end{align}
    In the last two inequalities, we have used \eqref{asi} and \eqref{ref1}.
   Thus, using \eqref{long} and \eqref{TN}, \eqref{aim} can be expressed as
 \begin{align}
     KT &\gtrsim_{n,\alpha,\beta,T^*} N^{1+\gamma}T^{1-\kappa} \approx N^{1+\gamma \left(1-\frac{\alpha (1-\kappa)} {\omega}\right)}\nonumber\\
    &\label{Npower}= N^{1-\gamma \left(-1+\frac{\alpha (1-\kappa)}{\omega}\right)}.
\end{align}
Note that $N$ can be chosen to be arbitrarily large. For any $\beta$ satisfying 
\begin{align}\label{betarange}
0 < \gamma <
\begin{cases}
 \frac{\omega}{\alpha (1-\kappa) -\omega} \quad &\text{if}\quad  \alpha (1-\kappa)-\omega>0\\
\infty \quad &\text{otherwise,} 
\end{cases}
\end{align}
the exponent of $N$ is positive in \eqref{Npower}, we get $KT>T^*$. This concludes the proof of Theorem \ref{gwp}.     
\end{proof}
\begin{Remark}\label{Whypso}
Recall $\gamma$ as defined in  \eqref{betap}.  The range of $\gamma$ in \eqref{betarange} decides the range of $p.$ Thus, we have $p \in (2,p_{max})$  with $p_{max}$ as given in \eqref{pmax}. This validates the choice of \(p\)  in Theorem \ref{gwp}.
\end{Remark}
\section{Proof of Theorem \ref{lwpHar} and Theorem \ref{gwpHar}}\label{s5}
\subsection{Local well-posedness in $L^2_{rad}+M^{\frac{4n}{2n-\nu},\frac{4n}{2n+\nu}}_{rad}$}\label{seclwp2}In this subsection, we shall prove Theorem \ref{lwpHar}. First, we introduce some notation and few lemmas before proceeding with the main proof.\\
Denote
\begin{equation}\label{theta}
    \theta:=\nu /\beta
\end{equation}
We define the Banach space $\tilde{X}(T)$ as
\begin{equation}
\begin{aligned}\label{tildeXT}
        \tilde{X}(T)&:= X_{3}(T)+X_{4}(T)
\end{aligned}
\end{equation}
where $$X_3(T):=C_{T}L^2_{rad} ~\cap~ L^{4/ {\theta}}_{T}L^{4n/(2n-\nu)}$$
equipped with the norm 
$$\|v\|_{X_3(T)}:=\max \left\{\|v\|_{L^{\infty}_{T}L^2},\|v\|_{L^{4/ {\theta}}_{T}L^{4n/(2n-\nu)}}\right\} $$ and $$ X_4(T):=C_{T}M^{\frac{4n}{2n-\nu},\frac{4n}{2n+\nu}}_{rad}.$$\\
The norm on $\tilde{X}(T)$ is given as 
\begin{equation*}
\|u\|_{\tilde{X}(T)} :=\inf_{\substack{u=v+w \\ v \in X_3(T) \\ w \in X_4(T)}} \left(\|v\|_{X_3(T)} + \|w\|_{X_4(T)} \right).
\end{equation*}
Denote $$Z(T):=L^{4/ {\theta}}_{T}L^{4n/(2n-\nu)}$$
\begin{Remark}\label{lemHarlwp2}
    Using H\"older's inequality for the time variable and Lemma \ref{srp}\eqref{embedding},  we have $\|\cdot\|_{Z(T)} \lesssim_{n} \|\cdot\|_{\tilde{X}(T)}$ for $T\leq 1.$ 
\end{Remark}
\begin{lem}\label{lemlwpHar}Let $n \geq 2, 0< \nu < \min \{\beta,n\}, \beta \in (\frac{2n}{2n-1}, 2)$ and $(q_{1},r_{1})\in \mathcal{A}_{\beta}$. Then, there exists a constant $C=C(n,\nu,\beta,r_{1})>0$
such that for any $T > 0$  
and $v,w_{1},w_{2} \in Z(T)$, the following estimate holds :
\begin{align*}
   \hspace{-0.5cm}\left|\left| \int_0^t U_\beta(t-s)\tilde{G}(v,w_{1},w_{2})(s) ds \right|\right|_{L^{q_{1}}_{T}L^{r_{1}}} &\lesssim_{n,\nu,\beta,r_{1}} T^{1-\theta}\|w_{1}-w_{2}\|_{Z(T)}\left(\|v\|_{Z(T)}^{2}
     + \|w_{1}\|_{Z(T)}^{2} + \|w_{2}\|_{Z(T)}^{2}\right.\nonumber\\
     +&\left. \|w_{1}\|_{Z(T)}\|w_{2}\|_{Z(T)}+\|v\|_{Z(T)}\|w_{1}\|_{Z(T)}+\|v\|_{Z(T)}\|w_{2}\|_{Z(T)}\right).
\end{align*}
\end{lem}
\begin{proof}
  Using Proposition \ref{fst}\eqref{fst2} for $(q_j, r_j) \in \mathcal{A}_{\beta}$, $j=1,2$, we have 
  \begin{align*}
   \left|\left| \int_0^t U_\beta(t-s)\tilde{G}(v,w_{1},w_{2})(s) ds \right|\right|_{L^{q_{1}}_{T}L^{r_{1}}} \lesssim_{n,\beta,r_{1},r_{2}}  \left|\left| \tilde{G}(v,w_{1},w_{2})\right|\right|_{L^{q'_{2}}_{T}L^{r'_{2}}} .
\end{align*}
Taking $(q_{2},r_{2})=(\frac{4}{\theta},\frac{4n}{2n-\nu})$ and using \eqref{haridentity}, H\"older's inequality twice and Lemma \ref{HLS}, we have
\begin{align*}
    \left|\left| \tilde{G}(v,w_{1},w_{2})\right|\right|_{L^{q'_{2}}_{T}L^{r'_{2}}} 
    \lesssim_{n,\alpha,\beta,r_{1}} &T^{1-\theta}\|w_{1}-w_{2}\|_{Z(T)}\left(\|v\|_{Z(T)}^{2}
     + \|w_{1}\|_{Z(T)}^{2} + \|w_{2}\|_{Z(T)}^{2}\right.\nonumber\\
     &+\left. \|w_{1}\|_{Z(T)}\|w_{2}\|_{Z(T)}+\|v\|_{Z(T)}\|w_{1}\|_{Z(T)}+\|v\|_{Z(T)}\|w_{2}\|_{Z(T)}\right).
\end{align*}
\end{proof}

\begin{proof}[\textbf{Proof of Theorem \ref{lwpHar}}] 
By Duhamel's principle, the equation \eqref{FNLSH} can be expressed as an integral equation:
\begin{eqnarray}\label{DPhar}
  u(x,t)=U_{\beta}(t)u_{0}(x) + i  \int_0^t U_\beta(t-s)\tilde{G}(0,u,0)(s) ds :=\Pi (u)(x,t).  
\end{eqnarray}
Let $a$ and $T$ be a positive real numbers (to be chosen later).  Define $$B(a,T)=\{u\in \tilde{X}(T):\|u\|_{\tilde{X}(T)}\leq a\}.$$
We shall show that $\Pi(u)$ (as defined in \eqref{DPhar}) is a contraction mapping on $B(a, T).$
Assume, without loss of generality, that $T\leq1.$ 
Consider an arbitrary decomposition of $u_{0}$ as $$u_{0}=v_{0}+w_{0}\in L^{2}_{rad}+ M^{\frac{4n}{2n-\nu},\frac{4n}{2n+\nu}}_{rad},$$
where $v_{0}\in L^{2}_{rad}$ and $w_{0}\in M^{\frac{4n}{2n-\nu},\frac{4n}{2n+\nu}}_{rad}.$\\

For the linear evolution of $u_{0},$ using Proposition \ref{fst}\eqref{fst2} and Proposition \ref{uf}, we have 
\begin{align*}
     \|U_{\beta}(t)u_{0} \|_{\tilde{X}(T)}&\leq  \|U_{\beta}(t)v_{0} \|_{X_{3}(T)}+\|U_{\beta}(t)w_{0} \|_{X_{4}(T)} \\ &\lesssim_{n,\nu,\beta} \|v_{0}\|_{L^{2}}+(1+T)^{n\left( \frac{1}{2}-\frac{2n-\nu}{4n} \right)} \|w_{0}\|_{M^{\frac{4n}{2n-\nu},\frac{4n}{2n+\nu}}}\\
&\lesssim \|v_{0}\|_{L^{2}}+\|w_{0}\|_{M^{\frac{4n}{2n-\nu},\frac{4n}{2n+\nu}}}.
\end{align*}
Since the decomposition is arbitrary, it follows that
    $$\|U_{\beta}(t)u_{0} \|_{\tilde{X}(T)} \lesssim_{n,\nu,\beta} \|u_{0}\|_{L^2+M^{\frac{4n}{2n-\nu},\frac{4n}{2n+\nu}}}.$$
This suggests the choice of $a=2C(n,\nu,\beta)\|u_{0}\|_{L^{2}+M^{\frac{4n}{2n-\nu},\frac{4n}{2n+\nu}}}.$ The integral part is estimated using $X_{3}(T)\hookrightarrow \tilde{X}(T),$ Lemma \ref{lemlwpHar} (with $r_{1} \in \{\frac{4n}{2n-\nu},2\}$) and Remark \ref{lemHarlwp2} as follows
    \begin{align}
    \left|\left| \int_{0}^{t} U_{\beta}(t-\tau)\tilde{G}(0,u,0)(\tau)d\tau \right|\right|_{\tilde{X}(T)}
    &\lesssim \left|\left| \int_{0}^{t} U_{\beta}(t-\tau)\tilde{G}(0,u,0)(\tau)d\tau \right|\right|_{X_{3}(T)} \nonumber\\
    &\lesssim_{n,\nu,\beta} T^{1-\theta} \|u\|_{Z(T)}^{3}\nonumber\\
    &\lesssim_{n} T^{1-\theta} \|u\|_{\tilde{X}(T)}^{3} \leq T^{1-\theta}a^{3}, \nonumber
    \end{align}
provided $u\in B(a,T)$. Taking 
    \begin{equation}\label{tildeTchoice}
        T :=\min \left\{ 1,~C(n,\nu,\beta)  \|u_{0}\|_{L^2+M^{\frac{4n}{2n-\nu},\frac{4n}{2n+\nu}}}^{-\frac{2}{1-\theta }} \right\}~,
    \end{equation}
    we conclude $\Pi(u)\in B(a, T).$ Similarly,  we can show that $\Pi(u)$ is a contraction mapping. Using Banach contraction mapping principle, we obtain a unique fixed point $u$ for $\Pi(u)$ which is a solution to \eqref{FNLSH} on the time interval $[0,T]$.
\end{proof}
\subsection{Global well-posedness in $M^{s,s'}_{rad}$}\label{secgwpHar} In this subsection, we shall prove that the local solution  $u$ obtained in Theorem \ref{lwpHar} can be extended globally in time when 
$u_{0} \in M^{s,s'}_{rad}$. We employ the same strategy used in the proof of Theorem \ref{gwp}.
Suppose to the contrary that the solution established in 
Theorem \ref{lwpHar} is not global in time and there exists the maximal time $T^*< \infty$. 

We firstly decompose initial data $u_0 \in M^{s,s'}_{rad} \subset L^2_{rad} + M^{\frac{4n}{2n-\nu},\frac{4n}{2n+\nu}}_{rad} $ into two parts using Lemma \ref{ipt} for radial functions. \\ For any $\tilde{N}>1$ and $u_0 \in M^{s,s'}_{rad},$ there exists  $\phi_0 \in L^2_{rad}$ and $ \psi_0 \in M^{\frac{4n}{2n-\nu},\frac{4n}{2n+\nu}}_{rad}$ 
such that 
\begin{equation}\label{dphar}
    u_0= \phi_0 + \psi_0
\end{equation}
with 
\begin{eqnarray}\label{asihar}
\|\phi_0\|_{L^2} \lesssim \tilde{N}^{\tilde{\gamma}},  \quad \|\psi_0\|_{M^{\frac{4n}{2n-\nu},\frac{4n}{2n+\nu}}} \lesssim  1/ \tilde{N}
\end{eqnarray}
where\begin{equation}\label{betaq}
    \tilde{\gamma} = \frac{\frac{1}{2} - \frac{1}{s}}{\frac{1}{s} - \frac{2n-\nu}{4n}}.
\end{equation}
Now, Consider \eqref{FNLSH} with initial data $\phi_0,$ namely
\begin{eqnarray}
\begin{cases}
 i \partial_t v_{0} - (-\Delta)^\frac{\beta}{2} v_{0}+(|\cdot|^{-\nu} \ast |v_{0}|^2)v_{0}=0 \\
v_{0}(\cdot,0)=\phi_{0}\in L^2_{rad}.   
\end{cases}        
    \end{eqnarray}
By \cite[Proposition 3.4]{DGBJDE}, \eqref{FNLSH} has a unique solution $v_{0}$ in the space
\begin{eqnarray}\label{gwpHarl2}
  C(\R, L^2_{rad}) \cap L_{loc}^{q}(\R,L^{r} )
\end{eqnarray} with $(q,r) \in \mathcal{A}_{\beta}$ and satisfies
\begin{equation}\label{v0qr2Har}
\sup_{(q,r)\in \mathcal{A}_{\beta}}\|v_{0}\|_{L_{loc}^{q} L^{r}} \lesssim_{n,r,\beta} \|\phi_{0}\|_{L^{2}}.
\end{equation}
Next, we consider the modified \eqref{FNLSH} corresponding to the evolution of $\psi_{0}$: 
\begin{eqnarray}\label{ivpModHar}
\begin{cases}
 i \partial_t w -(-\Delta)^\frac{\beta}{2} w +(|\cdot|^{-\nu} \ast |w+v_{0}|^2) (w+v_{0})-(|\cdot|^{-\nu} \ast |v_{0}|^2) v_{0}=0 \\
w(\cdot,0)=\psi_{0}\in M^{\frac{4n}{2n-\nu}, \frac{4n}{2n+\nu}}_{rad} .
\end{cases}
\end{eqnarray}
The solution to the above I.V.P. \eqref{ivpModHar} is given as
\begin{equation}
U_{\beta}(t)\psi_{0}+w_{0},
\end{equation}
where $w_0$ is the nonlinear interaction associated with $\psi_0,$ expressed as \begin{align}
w_{0} &= i\int_{0}^{t}  U_\beta(t-\tau)\left\{(|\cdot|^{-\nu} \ast |U_{\beta}(\tau)\psi_{0}+w_{0}+v_{0}|^2)(U_{\beta}(\tau)\psi_{0}+w_{0}+v_{0}) - (|\cdot|^{-\nu} \ast |v_{0}|^2) v_{0}\right\} \, d\tau \nonumber\\
    &=i\int_{0}^{t} U_\beta(t-\tau) \tilde{G}(v_{0} + U_{\beta}(\tau)\psi_{0}, w_{0},0) \, d\tau +  i \int_{0}^{t}  U_\beta(t-\tau) \tilde{G}(v_{0}, U_{\beta}(\tau)\psi_{0},0) \, d\tau\label{w01har}
\end{align}
Thus, the solution to the \eqref{FNLSH} having initial data $u_{0}$ (as given in \eqref{dphar}) can be written as:
\begin{eqnarray*}
v_{0}+U_{\beta}(t)\psi_{0}+w_{0}.
\end{eqnarray*}
 
 Since $v_0$ and $ U_{\beta}(t)\psi_0$ are globally defined in appropriate spaces as a result of  \eqref{gwpHarl2} and Proposition \ref{uf}, we are left to examine the time interval of existence for $w_0$ (given in  \eqref{w01har}). For this purpose, we establish local existence of \eqref{w01har} for general $\phi$ and $\psi$.
\begin{prop}\label{w0existhar} Let $\phi \in L^2_{rad}$ and $ \psi\in M^{\frac{4n}{2n-\nu},\frac{4n}{2n+\nu}}_{rad}$.  Assume that  $\beta$ and $\nu$  be  as in Theorem \ref{gwpHar}. Denote by  $v$  the $L^2-$global solution for initial value $\phi$.
Then  there exists  a constant $C=C(n, \nu,\beta)>0$ such that the integral equation
 $$w= i\int_{0}^{t} U_\beta(t-\tau) \tilde{G}(v + U_{\beta}(\tau)\psi, w,0) \, d\tau +  i \int_{0}^{t}  U_\beta(t-\tau) \tilde{G}(v, U_{\beta}(\tau)\psi,0) \, d\tau$$
     has a unique solution  $w \in Z(T)$   provided $T$ satisfying
\begin{align}
    \label{D1} T &\leq 1 \\
    \label{D2} T &\leq C \left( \|\phi\|_{L^{2}} + \|\psi\|_{M^{\frac{4n}{2n-\nu},\frac{4n}{2n+\nu}}} \right)^{-2 /(1-\theta)} \\
    \label{D3} T &\leq C \left( \|\psi\|_{M^{\frac{4n}{2n-\nu},\frac{4n}{2n+\nu}}} \right)^{-4/(2-\theta)}.
\end{align}
\end{prop}
\begin{proof}
     Define $$B(\tilde{A},T)=\{w\in Z(T):\|w\|_{Z(T)}\leq \tilde{A}\}$$
   with $\tilde{A}>0$ (to be choosen later). Let $T$ be the minimum of the right-hand sides of the conditions \eqref{D1},  \eqref{D2} and \eqref{D3}.
    Define the operator
   \begin{align*}
    \tilde{\Gamma}(w)=i\int_{0}^{t} U_\beta(t-\tau) \tilde{G}(v + U_{\beta}(\tau)\psi, w,0) \, d\tau +  i \int_{0}^{t}  U_\beta(t-\tau) \tilde{G}(v, U_{\beta}(\tau)\psi,0) \, d\tau. 
     \end{align*}
   Using Lemma \ref{lemlwpHar} and embedding $L^{\infty}_{T} \hookrightarrow L^{4/ {\theta}}_{T}$ under the assumption \eqref{D1}, for any $v_{1},v_{2} \in \C,$ we have 
    \begin{equation*}
    \left\| \int_{0}^{t} U_\beta(t-\tau)\tilde{G}(v_{1},v_{2},0) \, d\tau \right\|_{Z(T)}
     \end{equation*}
     \begin{align}
   \lesssim_{n,\nu,\beta}\label{beforehar4.10} & T^{1-\theta}\|v_{2}\|_{Z(T)}
     \left(
     \|v_{1}\|^{2}_{Z(T)}+\|v_{1}\|_{Z(T)} \|v_{2}\|_{Z(T)}+
     \|v_{2}\|^{2}_{Z(T)}\right)\\
      \label{har4.10}\lesssim  T^{1-\frac{3\theta}{4}}&\|v_{1}\|^{2}_{Z(T)}\|v_{2}\|_{L^{\infty}_{T}L^{\frac{4n}{2n-\nu}}}+T^{1-\frac{\theta}{2}}\|v_{1}\|_{Z(T)} \|v_{2}\|_{L^{\infty}_{T}L^{\frac{4n}{2n-\nu}}}^{2}+T^{1-\frac{\theta}{4}}
\|v_{2}\|^{3}_{L^{\infty}_{T}L^{\frac{4n}{2n-\nu}}}.
\end{align}
Using the estimate \eqref{har4.10} for $v_{1}=v , v_{2}=U_{\beta}(\tau)\psi$ along with \eqref{v0qr2Har}, Lemma \ref{srp} \eqref{embedding} and Proposition \ref{uf} under the assumption \eqref{D1}, we obtain
    \begin{align*} \hspace{-7cm}\left|\left| \displaystyle\int_{0}^{t} U_\beta(t-\tau)\tilde{G}(v,U_{\beta}(\tau)\psi,0) \, d\tau \right|\right|_{Z(T)} 
    \end{align*}
    \begin{align*}
&\hspace{-0.5cm}\lesssim_{\nu,n,\beta} T^{1-\frac{3\theta}{4}}\|v\|^{2}_{Z(T)}\|U_{\beta}(\tau)\psi\|_{L^{\infty}_{T}L^{\frac{4n}{2n-\nu}}}+T^{1-\frac{\theta}{2}}\|v\|_{Z(T)} \|U_{\beta}(\tau)\psi\|_{L^{\infty}_{T}L^{\frac{4n}{2n-\nu}}}^{2}+T^{1-\frac{\theta}{4}}
     \|U_{\beta}(\tau)\psi\|^{3}_{L^{\infty}_{T}L^{\frac{4n}{2n-\nu}}} \\
    &\hspace{-0.5cm} \lesssim_{\nu,n,\beta}  T^{1-\frac{3\theta}{4}} \| \phi \|_{L^{2}}^{2}  \| \psi\|_{M^{\frac{4n}{2n-\nu}, \frac{4n}{2n+\nu}}}+ T^{1-\frac{\theta}{2}}   \| \phi \|_{L^{2}} \| \psi \|_{M^{\frac{4n}{2n-\nu}, \frac{4n}{2n+\nu}}}^{2}+T^{1-\frac{\theta}{4}} \| \psi\|_{M^{\frac{4n}{2n-\nu}, \frac{4n}{2n+\nu}}}^{3}\\
    & \hspace{-0.5cm}\leq T^{\frac{\theta}{4}} \| \psi \|_{M^{\frac{4n}{2n-\nu}, \frac{4n}{2n+\nu}}}\left( T^{1-\theta}  (\| \phi \|_{L^{2}}+\| \psi \|_{M^{\frac{4n}{2n-\nu}, \frac{4n}{2n+\nu}}})^{2} 
      + T^{1-\frac{\theta}{2}}\| \psi \|_{M^{\frac{4n}{2n-\nu}, \frac{4n}{2n+\nu}}}^{2} \right)  \\
    &\hspace{-0.5cm} \lesssim_{\nu,n,\beta}    T^{\frac{\theta}{4}}
     \| \psi \|_{M^{\frac{4n}{2n-\nu}, \frac{4n}{2n+\nu}}}.
\end{align*}
The last inequality follows due to our assumptions \eqref{D2} and \eqref{D3}.
This suggests the choice of
    \begin{equation}\label{AB}
    \tilde{A}= \frac{5}{C(n,\nu,\beta)}T^{\frac{\theta}{4}}\|\psi\|_{M^{\frac{4n}{2n-\nu},\frac{4n}{2n+\nu}}}.
    \end{equation}
    where $C=C(n,\nu,\beta)$ is the same constant that appears in \eqref{D2} and \eqref{D3}.\\ Thus, \begin{equation}\label{gammaw1har}
      \left|\left| \displaystyle\int_{0}^{t} U_\beta(t-\tau)\tilde{G}(v,U_{\beta}(\tau)\psi,0)\, d\tau \right|\right|_{Z(T)}\leq \tilde{A}/5.
    \end{equation}
     Using the estimate \eqref{beforehar4.10} for $v_{1}=v + U_{\beta}(\tau)\psi$ and $ v_{2}=w$ along with \eqref{v0qr2Har}, Lemma \ref{srp} \eqref{embedding}  and Proposition \ref{uf} under the assumption \eqref{D1}, we have
\begin{align*}
    \hspace{-7cm} \left\| \int_{0}^{t}  U_\beta(t-\tau)\tilde{G}(v +  U_{\beta}(\tau)\psi, w,0) \, d\tau \right\|_{Z(T)} 
\end{align*}\ignorespaces
\begin{align*}
    & \lesssim_{\nu,n,\beta} T^{1-\theta}   \left\{ \|v
    + U_{\beta}(\tau)\psi \|_{Z(T)}^{2} \|w\|_{Z(T)} +\|v
    + U_{\beta}(\tau)\psi \|_{Z(T)} \|w\|_{Z(T)}^{2}+ \|w\|_{Z(T)}^{3} \right\} \\
    & \lesssim_{\nu,n,\beta}  \|w\|_{Z(T)}\left\{2T^{1-\theta}   \left( \|\phi\|_{L^{2}} + \|\psi\|_{M^{\frac{4n}{2n-\nu}, \frac{4n}{2n+\nu}}} \right)^{2}+2T^{1-\theta}  \|w\|_{Z(T)}^{2} \right\} \\
    & \lesssim_{\nu,n,\beta}\|w\|_{Z(T)} \left\{\frac{2}{5} +2T^{1-\frac{\theta}{2}}\|\psi\|^{2}_{M^{\frac{4n}{2n-\nu},\frac{4n}{2n+\nu}}}   \right\}.
\end{align*}
In the last inequality, we have used \eqref{D2} in the first summand and substitute the square norm of $w$ in $Z(T)$ by $\tilde{A}^{2}$, ($\tilde{A}$ given in \eqref{AB}) in the second summand. \\
Consider second summand of last inequality under the assumption \eqref{D3}  to get
        \begin{equation}\label{gamma2har}
            \left\| \int_{0}^{t} U_\beta(t-\tau)\tilde{G}(v + U_{\beta}(\tau)\psi, w,0)\, d\tau \right\|_{Z(T)} \leq 4\tilde{A} /5.
        \end{equation}
        Combining \eqref{gammaw1har} and \eqref{gamma2har}, we can say that $\tilde{\Gamma}(w)$ belongs to $B(\tilde{A}, T).$ Contractivity of $\tilde{\Gamma}$ follows similarly \footnote{$C(n,\nu,\beta)$ is choosen small enough such that all requirements are fulfilled.}. Thus, by the Banach fixed-point theorem, we get a unique fixed point $w$ to the integral equation \eqref{w01har} on the time-interval $[0,T].$
\end{proof}  
 \begin{cor}\label{winfty2har}
 Under the hypothesis of Proposition \ref{w0existhar}, there exists a constant $C(n,\nu,\beta)$ satisfying
    \begin{equation*}
\|w\|_{L^{\infty}_{T}L^{2}}\lesssim_{n,\nu,\beta} T^{\frac{\theta}{4}}\|\psi\|_{M^{\frac{4n}{2n-\nu},\frac{4n}{2n+\nu}}}.
    \end{equation*}
\end{cor}
\begin{proof}
The proof follows in a similar manner to Corollary \ref{winfty2}.
\end{proof}
\begin{proof}[\textbf{Proof of Theorem \ref{gwpHar}}]
Denote the constant from Proposition \ref{w0existhar} by $C=C(n,\nu,\beta)$ and put
\begin{equation}\label{TM}
        T=T(\tilde{N})=(3C\tilde{N}^{\tilde{\gamma}})^{-2/(1-\theta) }.
    \end{equation}
Applying Proposition \ref{w0existhar} for $\phi=\phi_{0}$ and $\psi=\psi_{0}$, we can say that the solution $u =v_0 + U_{\beta}(t) \psi_0 + w_0$ to \eqref{FNLSH} corresponding to the data $u_{0}$ in \eqref{dphar}
exists in the interval $[0, T(\tilde{N})]$.\\
We aim to extend our solution further by introducing the same iterative procedure as in the proof of Theorem \ref{gwp}, see \eqref{iter1}, but with Hartree nonlinearity. However, the nonlinear interaction term $w_{k-1}$ is given as: 
\begin{eqnarray}
    \begin{aligned}
     w_{k-1}= i \displaystyle \int_{0}^{t} U_{\beta}(t-\tau) \tilde{G}(v_{k-1}+U_{\beta}(k\tau)\psi_{0},w_{k-1},0) d\tau \nonumber \\
    + i \displaystyle \int_{0}^{t} U_{\beta}(t-\tau)\tilde{G}(v_{k-1},U_{\beta}(k\tau)\psi_{0},0) d\tau.
    \end{aligned}
       \end{eqnarray}
 We shall show that  the iterative process terminates with  $KT >T^*$ by extending the solution to the $K$th iteration.
We shall use Proposition \ref{w0existhar} with $\phi=\phi_{K}$ and $\psi=U_{\beta}(KT)\psi_{0}.$  

Taking into account \eqref{TM}, $T=T(\tilde{N})\to 0$ as $\tilde{N}\to \infty$. Thus, \eqref{D1} holds independently of $k$ for large $\tilde{N}.$ \\
Using Proposition \ref{uf} and  \eqref{asihar}, we have
\begin{align}\label{ref1har} 
\|U_{\beta}(t)\psi_{0}\|_{L^{\infty}([0,T^{*}+1],M^{\frac{4n}{2n-\nu},\frac{4n}{2n+\nu}}) } &\lesssim_{n,T^*}\|\psi_{0}\|_{M^{\frac{4n}{2n-\nu},\frac{4n}{2n+\nu}}} \lesssim_{n} 1 / \tilde{N}
\xrightarrow{\tilde{N} \to \infty} 0.
    \end{align}
Inserting $U_{\beta}(kT)\psi_{0}$ in the right hand side of \eqref{D3}, we have
\begin{eqnarray*}
\begin{aligned}
    \left( \|U_{\beta}(kT)\psi_{0}\|_{M^{\frac{4n}{2n-\nu},\frac{4n}{2n+\nu}}} \right)^{-4/(2-\theta)} \gtrsim_{n} \tilde{N}^{4/ (2-\theta)} \xrightarrow{\tilde{N} \to \infty} \infty.
   \end{aligned}
\end{eqnarray*}  
Thus, the condition \eqref{D3} holds for sufficiently large $\tilde{N}$ as the lower bound does not depend on $k$ and $T \xrightarrow{\tilde{N} \to \infty} 0.$ \\
Now, we are left with two cases. We either have $KT> T^*$ or the condition \eqref{D2} fails in the last iterative step $k=K,$ i.e.
\begin{equation}\label{aim11Har}
    3C\tilde{N}^{\tilde{\gamma}}< \|\phi_{K}\|_{L^2}+ \|U_{\beta}(KT)\psi_{0}\|_{M^{\frac{4n}{2n-\nu},\frac{4n}{2n+\nu}}}.
\end{equation} 
Considering \eqref{ref1har}, \eqref{aim11Har} can be written as
\begin{equation}\label{aimHar}
    3C\tilde{N}^{\tilde{\gamma}}< \|\phi_{K}\|_{L^2}+C\tilde{N}^{\tilde{\gamma}}.
\end{equation}
We shall prove that even under condition \eqref{aimHar}, we have $KT>T^*$ which clearly leads to a contradiction on $T^*.$ 

Using conservation of mass \eqref{mass} and Corollary \ref{winfty2har} (for  $w=w_{k}$ and $\psi=U_{\beta}(kT)\psi_{0}, \; 0\leq k \leq K-1$), we have
    \begin{align}
        \|\phi_{K}\|_{L^{2}}  &\leq \|v_{K-1}\|_{L^{\infty}_{[(K-1)T,KT]}L^{2}}+\|w_{K-1}\|_{L^{\infty}_{[(K-1)T,KT]}L^{2}} \nonumber\\
        & \lesssim_{n, \nu,\beta} \tilde{N}^{\tilde{\gamma}}+T^{\frac{\theta}{4}}\sum_{k=0}^{K-1}\|U_{\beta}(kT)\psi_{0}\|_{\tilde{N}^{\frac{4n}{2n-\nu},\frac{4n}{2n+\nu}}}\nonumber\\
        &\lesssim_{n,T^{*}} \tilde{N}^{\tilde{\gamma}}+T^{\frac{\theta}{4}}K\frac{1}{\tilde{N}}.
    \end{align}
    In the last two inequalities, we have used \eqref{asihar} and \eqref{ref1har}.
   Thus, using \eqref{TM}, \eqref{aimHar} can be expressed as
 \begin{align}
     KT &\gtrsim_{n,\nu,\beta,T^*} \tilde{N}^{1+\tilde{\gamma}}T^{1-\frac{\theta}{4}} \approx \tilde{N}^{1+\tilde{\gamma} \left(1-\frac{4-\theta}{2(1-\theta)}\right)}\nonumber\\
    &\label{Mpower}= \tilde{N}^{1 - \tilde{\gamma} \left( \frac{2 + \theta}{2(1 - \theta)} \right)}
\end{align}
By choosing the exponent of $\tilde{N}$ to be positive in \eqref{Mpower} and $\tilde{N}$ to be arbitrarily large, we get $KT>T^*$.
This concludes the proof of Theorem \ref{gwpHar}.
Note that the exponent of $\tilde{N}$ decides the range of $\tilde{\gamma}$ which in turn decides the range of $s$, see \eqref{betaq}. Thus, we have $s \in (2,s_{max})$  with $s_{max}$ as given in \eqref{qmax}. This validates the choice of \(s\)  in Theorem \ref{gwpHar}.
\end{proof}
{\bf Acknowledgments:} The authors would like to express their gratitude to professor H.G. Feichtinger   for his valuable comments and suggestions on an earlier version of this paper.
The second author acknowledges the financial support from the University Grants Commission (UGC), India (file number 201610135365). The third author would like to thankfully acknowledge the financial support from the Matrics Project of DST (file number 2018/001166).
\bibliographystyle{plain}
\bibliography{finls.bib}
\end{document}